\documentclass[12pt,reqno]{amsart}
\usepackage{txfonts,fullpage}
\usepackage[dvipdfmx]{graphicx}

\newtheorem{thm}{Theorem}[section]
\newtheorem{prop}[thm]{Proposition}

\newtheorem{cor}[thm]{Corollary}
\newtheorem{fact}[thm]{Fact}

\newtheorem{rem}[thm]{Remark}
\numberwithin{equation}{section}
\allowdisplaybreaks

\begin{document}

\title{Degenerations of the fourth order $q$-Garnier system arising from confluences in quivers}
\date{}
\author{Kazuya Matsugashita}
\address{Graduate School of Science and Engineering, Kindai University, 3-4-1, Kowakae, Higashi-Osaka, Osaka 577-8502, Japan}
\email{2444310104w@kindai.ac.jp}
\author{Takao Suzuki}
\address{Department of Mathematics, Kindai University, 3-4-1, Kowakae, Higashi-Osaka, Osaka 577-8502, Japan}
\email{suzuki@math.kindai.ac.jp}
\author{Satoshi Tsuchimi}
\address{Graduate School of Science and Engineering, Kindai University, 3-4-1, Kowakae, Higashi-Osaka, Osaka 577-8502, Japan}
\email{tsuchimi@math.kindai.ac.jp}

\maketitle

%%%%%%%%%%%%%%%%%%%%%%%%%%%%%%%%%%%%%%%%%%%%%%%%%%%%%%%%%%%%%%%%
\begin{abstract}

The $q$-Garnier system was first proposed by Sakai, and alternative directions for its discrete time evolutions were provided by Nagao and Yamada.
Recently, it was shown that all of those evolution directions are derived from a birational representation of an extended affine Weyl group which arises from the cluster algebraic construction established by Masuda, Okubo and Tsuda.
In this article, we investigate a degeneration structure of the fourth order $q$-Garnier system by using confluences in quivers.


\end{abstract}

%%%%%%%%%%%%%%%%%%%%%%%%%%%%%%%%%%%%%%%%%%%%%%%%%%%%%%%%%%%%%%%%
% Section 1
%%%%%%%%%%%%%%%%%%%%%%%%%%%%%%%%%%%%%%%%%%%%%%%%%%%%%%%%%%%%%%%%
\section{Introduction}

The $q$-Garnier system was first proposed by Sakai as a deformation problem of a system of linear $q$-difference equations in \cite{Sak2}.
Afterward, Nagao and Yamada provided alternative directions for its discrete time evolutions by using the Pade method in \cite{NY2}.
Recently, Okubo and one of the present authors derived all of those evolution directions from a birational representation of an extended affine Weyl group of type $(A_{2n+1}+A_1+A_1)^{(1)}$ in \cite{OS1,Suz3}.
At that time, the birational representation was formulated based on the cluster algebraic construction of Weyl groups established by Masuda, Okubo and Tsuda in \cite{MOT1}.
The aim of this article is to investigate a degeneration structure of the fourth order $q$-Garnier system by using the MOT construction and confluences in quivers.

The degeneration structure of the $q$-Painlev\'e equations of second order is given as follows (see \cite{Sak1}).
\[
	\begin{array}{l@{\ }l@{\ }l@{\ }l@{\ }l@{\ }l@{\ }l@{\ }l@{\ }l@{\ }l@{\ }l@{\ }l@{\ }l@{\ }l@{\ }l}
		& & & & & & & & & & & & & & (A'_1)^{(1)} \\
		& & & & & & & & & & & & & \nearrow & \\
		E_8^{(1)} & \to & E_7^{(1)} & \to & E_6^{(1)} & \to & D_5^{(1)} & \to & A_4^{(1)} & \to & (A_2+A_1)^{(1)} & \to & (A_1+A'_1)^{(1)} & \to & A_1^{(1)} \\
	\end{array}
\]
The $q$-Garnier system is regarded as a generalization of the $q$-Painlev\'e equation of type $D_5^{(1)}$.
Bershtein and his collaborators derived all of the $q$-Painlev\'e equations by using the cluster mutation in \cite{BGM,BGMS}.
They also mentioned that the degeneration structure arises from confluences in quivers in \S5 of \cite{BGM}.
Afterward, Okubo and one of the present authors gave its explicit calculation in \cite{OS2}.
In this article, we apply those previous works to the fourth order $q$-Garnier system.
According to \cite{OS1}, it is derived from the quiver with $12$ vertices (Figure \ref{Fig:q-Garnier}).
We consider confluences from this quiver to those with $11$ vertices (Figure \ref{Fig:deg_q-Garnier_11}) and $10$ vertices (Figures \ref{Fig:deg_q-Garnier_101}, \ref{Fig:deg_q-Garnier_102}, \ref{Fig:deg_q-Garnier_103}, \ref{Fig:deg_q-Garnier_104} and \ref{Fig:deg_q-Garnier_105}).

Another aim is to investigate particular solutions to the degenerate $q$-Garnier systems.
It was shown in \cite{IS,NY1,Sak3,Suz2} that the $q$-Garnier system admits particular solutions in terms of the basic hypergeometric series ${}_{n+1}\phi_n$ or the $q$-Lauricella series $\phi_D^{(n)}$.
We give particular solutions to some of the obtained degenerate $q$-Garnier systems.
They are expressed as ratios of the basic hypergeometric series ${}_2\phi_2$ and ${}_1\phi_2$.

This article is organized as follows.
In Section \ref{Sec:Cluster}, we recall the definition of the cluster mutation and the confluence in a quiver.
In Section \ref{Sec:Affine_Weyl}, we introduce the MOT construction of a birational representation for the affine Weyl group of type $(A_5+A_1+A_1)^{(1)}$.
In Section \ref{Sec:q-Garnier}, the fourth order $q$-Garnier system is introduced together with its $q$-hypergeometric solution.
The results of this article are given in the following sections.
In Sections \ref{Sec:Degeneration_12_11} and \ref{Sec:Degeneration_11_10}, we derive degenerated $q$-Garnier systems by using the MOT construction and the confluences in quivers.
In Section \ref{Sec:HG_Solution}, we investigate particular solutions to some degenerate $q$-Garnier systems.

\begin{rem}
The birational representation of the affine Weyl group of type $(A_{2n+1}+A_1+A_1)^{(1)}$ is also formulated as a deformation problem of a system of linear $q$-difference equations with $(2n+2)\times(2n+2)$ matrices called a Lax form in \cite{Mas2,Suz3}.
We expect that the degeneration structure of the $q$-Garnier system is lifted to the Lax form, but an explicit calculation has not been given yet.
\end{rem}

\begin{rem}
In this article, we consider one quiver with $11$ vertices and five quivers with $10$ vertices as equivalence classes for the cluster mutations.
Moreover, we narrow down the candidates of the quiver with $9$ vertices to $7$ quivers by using software.
However it has not been clarified yet whether the candidates can be further narrowed down.
\end{rem}

\begin{rem}
Masuda, Okubo and Tsuda gave generalizations of the $q$-Painlev\'e equations of type $(A_1+A'_1)^{(1)}$ and $A_1^{(1)}$ based on their construction in \cite{MOT2}.
Their corresponding quivers have $7$ and $6$ vertices respectively and are obtained through confluence procedures as follows.
\begin{align*}
	&Q_{101}\ \xrightarrow{8\to9}\ \xrightarrow{6\to7}\ \xrightarrow{8\to4}\ (\text{Quiver for type $(A_1+A'_1)^{(1)}$}), \\
	&Q_{101}\ \xrightarrow{8\to9}\ \xrightarrow{6\to7}\ \xrightarrow{4\to6}\ \xrightarrow{1\to4} (\text{Quiver for type $A_1^{(1)}$}).
\end{align*}
Here the symbol $Q_{101}$ is a quiver with $10$ vertices given in Figure \ref{Fig:deg_q-Garnier_101} and the symbol $i\to j$ is a confluence in a quiver given in Section \ref{Sec:Cluster}.
To be precise, we need to take the mutation $\mu_3$ and a certain permutation after confluences in each case.
We have not clarified a relationship between those two quivers yet.
\end{rem}

\begin{rem}
Kawakami proposed a classification of fourth order Painlev\'e type difference equations in \cite{Kaw}.
Besides, Kawakami, Nakamura and Sakai classified the fourth order Painlev\'e type differential equations from a viewpoint of the deformation problems (see \cite{HKNS}).
The relationship between these previous works and the result of this article seems to be important, but this remains a subject for future research.
\end{rem}

\begin{rem}
It remains an open problem whether the $q$-Garnier system is in one-to-one correspondence with an equivalence class of quivers.
Besides, Masuda proposed another generalization of the $q$-Painlev\'e equation of type $D_5^{(1)}$ in \cite{Mas1}, which is known as the $q$-Sasano system and derived from a quiver distinct from $Q_{12}$ (see \cite{MOT1}).
There is a possibility that an evolution direction of the $q$-Garnier system is derived from the $q$-Sasano quiver.
Furthermore, the degeneration of the $q$-Sasano system remains a subject for future research.
\end{rem}

%%%%%%%%%%%%%%%%%%%%%%%%%%%%%%%%%%%%%%%%%%%%%%%%%%%%%%%%%%%%%%%%
% Section 2
%%%%%%%%%%%%%%%%%%%%%%%%%%%%%%%%%%%%%%%%%%%%%%%%%%%%%%%%%%%%%%%%
\section{Cluster mutation and confluence in quiver}\label{Sec:Cluster}

Let $Q$ be a quiver without $1$-loop or $2$-cycle and $I$ be a set of vertices of $Q$.
Then the quiver $Q$ is in one-to-one correspondence with the skew-symmetric matrix $\Lambda=\left(\lambda_{i,j}\right)_{i,j\in I}$ as follows.
\begin{itemize}
\item
If there are $N$ arrows from $i\in I$ to $j\in I$, then we set $\lambda_{i,j}=N,\ \lambda_{j,i}=-N$.
\item
If there are no arrows between $i\in I$ and $j\in I$, then we set $\lambda_{i,j}=\lambda_{j,i}=0$.
\end{itemize}
Also let $y=\left(y_i\right)_{i\in I}$ be a set of variables called the coefficients.

We define a mutation $\mu_i$ for each $i\in I$ as follows.
\begin{enumerate}
\item
If there exist $M$ arrows from $k\in I$ to $i\in I$ and $N$ arrows from $i$ to $l\in I$, then we add $MN$ arrows from $k$ to $l$.
\item
If $2$-cycles appear, then we remove all of them.
\item
We reverse directions of all arrows touching $i$.
\end{enumerate}
Then each mutation $\mu_i$ is interpreted as a transformation $(\Lambda,y)\mapsto(\tilde{\Lambda},\tilde{y})$ given by
\begin{align*}
	&\tilde{\lambda}_{k,l} = \left\{\begin{array}{ll}-\lambda_{k,l}&(\text{$k=i$ or $l=i$})\\\lambda_{k,l}+\frac{\mathrm{sgn}(\lambda_{k,i})+\mathrm{sgn}(\lambda_{i,l})}{2}\lambda_{k,i}\lambda_{i,l}&(\text{otherwise})\end{array}\right., \\
	&\tilde{y}_k = \left\{\begin{array}{ll}y_i^{-1}&(k=i)\\y_k\left(1+y_i^{\mathrm{sgn}(\lambda_{k,i})}\right)^{\lambda_{k,i}}&(k\neq i)\end{array}\right..
\end{align*}
We also introduce a permutation of two vertices $i,j\in I$ and denote it by $(i,j)$.
It is interpreted as a transformation $(\Lambda,y)\mapsto(\tilde{\Lambda},\tilde{y})$ given by
\[
	\tilde{\lambda}_{k,l} = \lambda_{(i,j)(k),(i,j)(l)},\quad
	\tilde{y}_k = y_{(i,j)(k)}.
\]
Moreover, we introduce a reversal of all vertices and denote it by $\iota$.
It is interpreted as a transformation $(\Lambda,y)\mapsto(\tilde{\Lambda},\tilde{y})$ given by
\[
	\tilde{\lambda}_{k,l} = -\lambda_{k,l},\quad
	\tilde{y}_k = y_k^{-1}.
\] 

\begin{figure}[h]
	\begin{center}
	%\fbox{
	\begin{picture}(50,55)
	\put(0,43){$4$}\put(38,46){\vector(-1,0){25}}\put(45,43){$1$}
	\put(3,38){\vector(0,-1){25}}\put(13,13){\vector(1,1){25}}\put(38,13){\vector(-1,1){25}}\put(48,13){\vector(0,1){25}}
	\put(0,0){$3$}\put(13,3){\vector(1,0){25}}\put(45,0){$2$}
	\end{picture}
	%}
	\qquad\raisebox{25pt}{$\xrightarrow{4\to1}$}\qquad
	%\fbox{
	\begin{picture}(50,55)
	\put(0,43){$1$}
	\put(39,14){\vector(-1,1){25}}\put(37,12){\vector(-1,1){25}}
	\put(0,0){$3$}\put(13,3){\vector(1,0){25}}\put(45,0){$2$}
	\end{picture}
	%}
	\caption{Confluence in quiver}\label{Fig:Confluence}
	\end{center}
\end{figure}
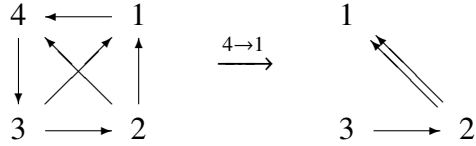

Following the previous works \cite{BGM,OS2}, we consider a confluence of two vertices $i,j\in I$.
It is denoted by $i\to j$ and defined as follows.
\begin{enumerate}
\item
Remove all arrows between $i$ and $j$.
\item
Glue $i$ and $j$ together with arrows touching these vertices (naming the new vertex $j$).
\item
Remove all of the 2-cycles that appear.
\end{enumerate}
The confluence $i\to j$ is interpreted as the following operation on the skew-symmetric matrix $\Lambda$.
\begin{enumerate}
\item
Add the $i$-th row to the $j$-th row and the $i$-th column to the $j$-th column.
\item
Decrease the matrix size by $1$ by deleting the $i$-th row and the $i$-th column.
\end{enumerate}
It is also defined as the following operation on the coefficients.
\begin{enumerate}
\item
Replace $y_i\to\varepsilon^{-1}y_j$ and $y_j\to\varepsilon$.
\item
Take a limit $\varepsilon\to0$.
\end{enumerate}
We give an example of the confluence in Figure \ref{Fig:Confluence}.
Then the corresponding skew-symmetric matrix is transformed as
\[
	\begin{pmatrix}0&-1&-1&1\\1&0&-1&1\\1&1&0&-1\\-1&-1&1&0\end{pmatrix}\quad
	\xrightarrow{4\to1}\quad
	\begin{pmatrix}0&-2&0\\2&0&-1\\0&1&0\end{pmatrix}
\]
and the coefficients are transformed as
\[
	(y_1y_4,y_2,y_3)\quad \xrightarrow{4\to1}\quad (y_1,y_2,y_3).
\]

%%%%%%%%%%%%%%%%%%%%%%%%%%%%%%%%%%%%%%%%%%%%%%%%%%%%%%%%%%%%%%%%
% Section 3
%%%%%%%%%%%%%%%%%%%%%%%%%%%%%%%%%%%%%%%%%%%%%%%%%%%%%%%%%%%%%%%%
\section{Birational representation of affine Weyl group}\label{Sec:Affine_Weyl}

Let $A=\left(a_{i,j}\right)_{i,j=0}^{N}$ be the generalized Cartan matrix defined by
\[
	a_{0,0} = a_{1,1} = 2,\quad
	a_{0,1} = a_{1,0} = -2
\]
for $N=1$ and
\[
	a_{i,j} = \left\{\begin{array}{ll}2&(i=j)\\-1&(j=i\pm1,i\neq0,N)\\-1&((i,j)=(0,1),(0,N),(N,0),(N,N-1))\\0&(\text{otherwise})\end{array}\right.
\]
for $N\geq2$.
The affine Weyl group of type $A^{(1)}_N$ is defined by the generators $r_i\ (i=0,\ldots,N)$ with the fundamental relation
\[
	r_0^2 = r_1^2 = 1
\]
for $N=1$ and
\[
	r_i^2 = 1,\quad
	(r_ir_j)^{2-a_{i,j}} = 1\quad (i,j=0,\ldots,N,\ i\neq j)
\]
for $N\geq2$.

\begin{figure}[h]
	\begin{center}
	%\fbox{
	\begin{picture}(135,100)
		\put(85,67.5){\small$1$}\put(105,90){\small$2$}
		\put(128,45){\small$3$}\put(105,45){\small$4$}
		\put(85,22.5){\small$5$}\put(105,0){\small$6$}
		\put(28,0){\small$7$}\put(52,22.5){\small$8$}
		\put(28,45){\small$9$}\put(1,45){\small$10$}
		\put(25,90){\small$11$}\put(50,67.5){\small$12$}
		\put(100,93.5){\vector(-1,0){60}}%2->11
		\put(40,90){\vector(3,-1){41}}%11->1
		\put(64,77){\vector(3,1){36}}%12->2
		\put(81,71){\vector(-1,0){17}}%1->12
		\put(23,86){\vector(-1,-2){15}}%11->10
		\put(30.5,56){\vector(0,1){30}}%9->11
		\put(48,65){\vector(-1,-1){11}}%12->9
		\put(15,55){\vector(2,1){31}}%10->12
		\put(103,55){\vector(-1,1){10}}%4->1
		\put(108,86){\vector(0,-1){29}}%2->4
		\put(128,56){\vector(-1,2){15}}%3->2
		\put(95,70){\vector(2,-1){29}}%1->3
		\put(47,27){\vector(-2,1){32}}%8->10
		\put(9,41){\vector(1,-2){15}}%10->7
		\put(30.5,11){\vector(0,1){30}}%7->9
		\put(37,43){\vector(1,-1){11}}%9->8
		\put(124,44){\vector(-2,-1){30}}%3->5
		\put(93,33){\vector(1,1){10}}%5->4
		\put(108,41){\vector(0,-1){30}}%4->6
		\put(114,12){\vector(1,2){14}}%6->3	
		\put(61,26){\vector(1,0){20}}%8->5
		\put(81,21){\vector(-3,-1){41}}%5->7
		\put(100,8){\vector(-3,1){39}}%6->8
		\put(40,3){\vector(1,0){60}}%7->6
	\end{picture}
	%}
	\caption{Quiver $Q_{12}$}\label{Fig:q-Garnier}
	\end{center}
\end{figure}
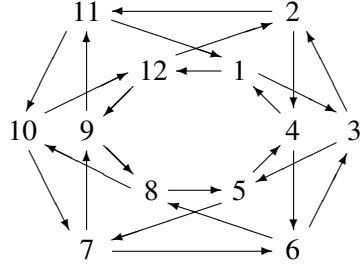

Based on the MOT construction in \cite {MOT1}, we formulate a group of birational transformations on $\mathbb{C}(y_i)$ which is isomorphic to the affine Weyl group of type $(A_5+A_1+A_1)^{(1)}$.
Let us consider the quiver $Q_{12}$ in Figure \ref{Fig:q-Garnier}.
Although this quiver has no $2$-cycles, it is considered to have $2$-cycles
\[
	2i+1 \to 2i+2 \to 2i+1\quad (i=0,\ldots,5).
\]
It also has $6$-cycles
\begin{align*}
	&1 \to 4 \to 5 \to 8 \to 9 \to 12 \to 1,\quad
	2 \to 3 \to 6 \to 7 \to 10 \to 11 \to 2, \\
	&1 \to 11 \to 9 \to 7 \to 5 \to 3 \to 1,\quad
	2 \to 12 \to 10 \to 8 \to 6 \to 4 \to 2.
\end{align*}
Note that, for any cycle $C$ in the above and any vertex $v$ outside of $C$, the number of arrows from $v$ to $C$ is equal to that from $C$ to $v$.
According to those cycles, we introduce compositions of mutations and permutations by
\begin{align*}
	&r_i = \mu_{2i+1}(2i+1,2i+2)\mu_{2i+1}\quad (i=0,\ldots,5), \\
	&s_0 = \mu_1\mu_4\mu_5\mu_8\mu_9(9,12)\mu_9\mu_8\mu_5\mu_4\mu_1,\quad
	s_1 = \mu_2\mu_3\mu_6\mu_7\mu_{10}(10,11)\mu_{10}\mu_7\mu_6\mu_3\mu_2, \\
	&s'_0 = \mu_1\mu_3\mu_5\mu_7\mu_9(9,11)\mu_9\mu_7\mu_5\mu_3\mu_1,\quad
	s'_1 = \mu_2\mu_4\mu_6\mu_8\mu_{10}(10,12)\mu_{10}\mu_8\mu_6\mu_4\mu_2.
\end{align*}
Note that the resulting elements do not depend on the orientations of the cycles.
Then they preserve the form of $Q_{12}$ and give birational transformations on $\mathbb{C}(y_i)$, which act on the coefficients as
\begin{align*}
	&r_i(y_{2i-1}) = \frac{y_{2i-1}y_{2i+2}(1+y_{2i+1})}{1+y_{2i+2}},\quad
	r_i(y_{2i}) = \frac{y_{2i}y_{2i+1}(1+y_{2i+2})}{1+y_{2i+1}}, \\
	&r_i(y_{2i+1}) = \frac{1}{y_{2i+2}},\quad
	r_i(y_{2i+2}) = \frac{1}{y_{2i+1}}, \\
	&r_i(y_{2i+3}) = \frac{y_{2i-1}y_{2i+3}(1+y_{2i})}{1+y_{2i-1}},\quad
	r_i(y_{2i+4}) = \frac{y_{2i}y_{2i+4}(1+y_{2i-1})}{1+y_{2i}}, \\
	&r_i(y_j) = y_j\quad (j\neq2i-1,\ldots,2i+4)
\end{align*}
and
\begin{align*}
	&s_0(y'_{2i-1}) = \frac{Y'_{2i-1}}{y'_{2i-3}Y'_{2i-5}},\quad
	s_0(y'_{2i}) = \frac{y'_{2i-3}y'_{2i-1}y'_{2i}Y'_{2i-5}}{Y'_{2i-1}}, \\
	&s_1(y'_{2i-1}) = \frac{y'_{2i-2}y'_{2i-1}y'_{2i}Y'_{2i-4}}{Y'_{2i}},\quad
	s_1(y'_{2i}) = \frac{Y'_{2i}}{y'_{2i-2}Y'_{2i-4}}, \\
	&s'_0(y'_{2i-1}) = \frac{Y_{2i-1}}{y_{2i+1}Y_{2i-5}},\quad
	s'_0(y'_{2i}) = \frac{y_{2i-1}y_{2i}y_{2i+1}Y_{2i-5}}{Y_{2i-1}}, \\
	&s'_1(y'_{2i-1}) = \frac{y_{2i-1}y_{2i}y_{2i+2}Y_{2i+4}}{Y_{2i}},\quad
	s'_1(y'_{2i}) = \frac{Y_{2i}}{y_{2i+2}Y_{2i+4}}
\end{align*}
for $i=0,\ldots,5$, where
\[
	Y_i = 1 + \sum_{j=1}^{5}\prod_{k=0}^{j-1}y_{i+2k},\quad
	Y'_i =1 +  \sum_{j=1}^{5}\prod_{k=0}^{j-1}y'_{i-2k}
\]
and
\[
	(y'_1,y'_2,y'_3,y'_4,y'_5,y'_6,y'_7,y'_8,y'_9,y'_{10},y'_{11},y'_{12}) = (y_1,y_2,y_4,y_3,y_5,y_6,y_8,y_7,y_9,y_{10},y_{12},y_{11}).
\]
Here the indices of the coefficients are congruent modulo $12$.
For example, the birational transformation $r_0$ is given as
\begin{align*}
	&\left(y_1,y_2,y_3,y_4,y_5,y_6,y_7,y_8,y_9,y_{10},y_{11},y_{12}\right) \\
	&\overset{\mu_1}{\mapsto}\ \left(y_1^{-1},y_2,\frac{y_3}{1+y_1^{-1}},(1+y_1)y_4,y_5,y_6,y_7,y_8,y_9,y_{10},(1+y_1)y_{11},\frac{y_{12}}{1+y_1^{-1}}\right) \\
	&\overset{(1,2)}{\mapsto}\ \left(y_2,y_1^{-1},\frac{y_3}{1+y_1^{-1}},(1+y_1)y_4,y_5,y_6,y_7,y_8,y_9,y_{10},(1+y_1)y_{11},\frac{y_{12}}{1+y_1^{-1}}\right) \\
	&\overset{\mu_1}{\mapsto}\ \left(y_2^{-1},y_1^{-1},\frac{y_3(1+y_2)}{1+y_1^{-1}},\frac{y_4(1+y_1)}{1+y_2^{-1}},y_5,y_6,y_7,y_8,y_9,y_{10},\frac{y_{11}(1+y_1)}{1+y_2^{-1}},\frac{y_{12}(1+y_2)}{1+y_1^{-1}}\right).
\end{align*}

We also introduce new variables by
\begin{align*}
	&\alpha_i = y_{2i+1}y_{2i+2}\quad (i=0,\ldots,5), \\
	&\beta_0 = y_1y_4y_5y_8y_9y_{12},\quad
	\beta_1 = y_2y_3y_6y_7y_{10}y_{11},\quad
	\beta'_0 = y_1y_3y_5y_7y_9y_{11},\quad
	\beta'_1 = y_2y_4y_6y_8y_{10}y_{12}
\end{align*}
and set
\[
	q = \prod_{i=0}^{5}\alpha_i = \beta_0\beta_1 = \beta'_0\beta'_1 = \prod_{i=1}^{12}y_i.
\]
Then the birational transformations act on them as
\begin{align*}
	&r_i(\alpha_j) = \alpha_j\alpha_i^{-a_{i,j}},\quad
	r_i(\beta_k) = \beta_k,\quad
	r_i(\beta'_k) = \beta'_k,\quad
	r_i(q) = q, \\
	&s_k(\alpha_j) = \alpha_j,\quad
	s_k(\beta_l) = \beta_l\beta_k^{-b_{k,l}},\quad
	s_k(\beta'_l) = \beta'_l,\quad
	s_k(q) = q, \\
	&s'_k(\alpha_j) = \alpha_j,\quad
	s'_k(\beta_l) = \beta_l,\quad
	s'_k(\beta'_l) = \beta'_l(\beta'_k)^{-b_{k,l}},\quad
	s'_k(q) = q
\end{align*}
for $i,j=0,\ldots,5$ and $k,l=0,1$, where $\left(a_{i,j}\right)_{i,j=0}^{5}$ and $\left(b_{k,l}\right)_{k,l=0}^{1}$ are the generalized Cartan matrices of type $A^{(1)}_5$ and $A^{(1)}_1$ respectively.
It is equivalent to the action of the simple reflections on the simple roots and the null root.
As can be seen from this, the following fact holds.

\begin{fact}[\cite{MOT1}]
The groups $\langle r_0,\ldots,r_5\rangle$, $\langle s_0,s_1\rangle$ and $\langle s'_0,s'_1\rangle$ are isomorphic to the affine Weyl groups of type $A^{(1)}_5$, $A^{(1)}_1$ and $A^{(1)}_1$ respectively.
Moreover, any two groups are mutually commutative.
\end{fact}

In the last, we consider compositions of permutations and a reversal by
\begin{align*}
	&\pi_1 = (2,4,6,8,10,12)(1,3,5,7,9,11), \\
	&\pi_2 = (1,2)(3,4)(5,6)(7,8)(9,10)(11,12), \\
	&\pi_3 = \iota(1,12)(2,11)(3,10)(4,9)(5,8)(6,7),
\end{align*}
where $(1,3,5,7,9,11)$ is the cyclic permutation given by
\[
	(1,3,5,7,9,11) = (1,3)(1,5)(1,7)(1,9)(1,11).
\]
Then they preserve the form of $Q_{12}$ and give birational transformations on $\mathbb{C}(y_i)$, which act on the coefficients, the simple roots and the null root as
\begin{align*}
	&\pi_1(y_i) = y_{i+2}\quad(i=1,\ldots,10),\quad
	\pi_1(y_{11}) = y_1,\quad
	\pi_1(y_{12}) = y_2, \\
	&\pi_2(y_{2i-1}) = y_{2i},\quad
	\pi_2(y_{2i}) = y_{2i-1}\quad (i=1,\ldots,6), \\
	&\pi_3(y_i) = \frac{1}{y_{13-i}}\quad (i=1,\ldots,12)
\end{align*}
and
\begin{align*}
	&\pi_1(\alpha_i) = \alpha_{i+1}\quad (i=0,\ldots,4),\quad
	\pi_1(\alpha_5) = \alpha_0, \\
	&\pi_1(\beta_0) = \beta_1,\quad
	\pi_1(\beta_1) = \beta_0,\quad
	\pi_1(\beta'_0) = \beta'_0,\quad
	\pi_1(\beta'_1) = \beta'_1,\quad
	\pi_1(q) = q, \\
	&\pi_2(\alpha_i) = \alpha_i\quad (i=0,\ldots,5), \\
	&\pi_2(\beta_0) = \beta_1,\quad
	\pi_2(\beta_1) = \beta_0,\quad
	\pi_2(\beta'_0) = \beta'_1,\quad
	\pi_2(\beta'_1) = \beta'_0,\quad
	\pi_2(q) = q, \\
	&\pi_3(\alpha_i) = \frac{1}{\alpha_{5-i}}\quad (i=0,\ldots,5), \\
	&\pi_3(\beta_0) = \frac{1}{\beta_0},\quad
	\pi_3(\beta_1) = \frac{1}{\beta_1},\quad
	\pi_3(\beta'_0) = \frac{1}{\beta'_1},\quad
	\pi_3(\beta'_1) = \frac{1}{\beta'_0},\quad
	\pi_3(q) = \frac{1}{q}.
\end{align*}
Those transformations correspond to those called the Dynkin diagram automorphisms.
The group generated by the simple reflections and the Dynkin diagram automorphisms is decomposed into the semi-direct product as
\[
	\langle r_0,\ldots,r_5,s_0,s_1,s'_0,s'_1,\pi_1,\pi_2,\pi_3\rangle = \langle r_0,\ldots,r_5,s_0,s_1,s'_0,s'_1\rangle \rtimes \langle\pi_1,\pi_2,\pi_3\rangle.
\]
Hence we call the group $\langle r_0,\ldots,r_5,s_0,s_1,s'_0,s'_1,\pi_1,\pi_2\rangle$ the extended affine Weyl group of type $(A_5+A_1+A_1)^{(1)}$.

%%%%%%%%%%%%%%%%%%%%%%%%%%%%%%%%%%%%%%%%%%%%%%%%%%%%%%%%%%%%%%%%
% Section 4
%%%%%%%%%%%%%%%%%%%%%%%%%%%%%%%%%%%%%%%%%%%%%%%%%%%%%%%%%%%%%%%%
\section{$q$-Garnier system}\label{Sec:q-Garnier}

In the following, we regard the simple reflections and the Dynkin diagram automorphisms as automorphisms of $\mathbb{C}(y_i)$.
For example, the composition $r_0r_1$ acts on the simple root $\alpha_0$ as
\[
	r_0r_1(\alpha_0) = r_0(\alpha_0\alpha_1) = r_0(\alpha_0)r_0(\alpha_1) = \alpha_1.
\]

Following the previous work \cite{OS1}, we consider elements of the extended affine Weyl group of type $(A_5+A_1+A_1)^{(1)}$
\begin{align*}
	&T_i = r_ir_{i+1}r_{i+2}r_{i+3}r_{i+4}r_{i+5}r_{i+4}r_{i+3}r_{i+2}r_{i+1}\quad (i=0,\ldots,5), \\
	&U_0 = s_0s_1,\quad
	U_1 = s_1s_0,\quad
	U'_0 = s'_0s'_1,\quad
	U'_1 = s'_1s'_0,\quad
	V = \pi_1r_5\ldots r_1s_1,\quad
	V' = \pi_2\pi_1r_5\ldots r_1s'_1,
\end{align*}
where the indices of $r_i$ are congruent modulo $6$.
Then they are called the translations and act on the simple roots as
\begin{align*}
	&T_i(\alpha_j) = q^{-\delta_{j,i-1}+2\delta_{j,i}-\delta_{j,i+1}}\alpha_j,\quad
	T_i(\beta_l) = \beta_l,\quad
	T_i(\beta'_l) = \beta'_l, \\
	&U_k(\alpha_j) = \alpha_j,\quad
	U_k(\beta_l) = q^{4\delta_{l,k}-2}\beta_l,\quad
	U_k(\beta'_l) = \beta'_l, \\
	&U'_k(\alpha_j) = \alpha_j,\quad
	U'_k(\beta_l) = \beta_l,\quad
	U'_k(\beta'_l) = q^{4\delta_{l,k}-2}\beta'_l, \\
	&V(\alpha_j) = q^{\delta_{j,0}-\delta_{j,1}}\alpha_j,\quad
	V(\beta_l) = q^{\delta_{l,0}-\delta_{l,1}}\beta_l,\quad
	V(\beta'_l) = \beta'_l, \\
	&V'(\alpha_j) = q^{\delta_{j,0}-\delta_{j,1}}\alpha_j,\quad
	V'(\beta_l) = \beta_l,\quad
	V'(\beta'_l) = q^{\delta_{l,0}-\delta_{l,1}}\beta'_l
\end{align*}
for $i,j=0,\ldots,5$ and $k,l=0,1$, where $\delta_{i,j}$ is the Kronecker delta whose indices are congruent modulo $6$.
Those translations generate an abelian normal subgroup of the extended affine Weyl group of type $(A_5+A_1+A_1)^{(1)}$.
This subgroup provides a system of fourth order non-linear $q$-difference equations with multiple evolution directions.
We call this system the fourth order $q$-Garnier system in this article.

We choose a translation
\[
	\tau_c = (V')^{-1}VU_1 = \pi_2s'_0s_0
\]
which acts on the simple roots as
\[
	\tau_c(\alpha_i) = \alpha_i\quad (i=0,\ldots,5),\quad
	\tau_c(\beta_0) = q^{-1}\beta_0,\quad
	\tau_c(\beta_1) = q\beta_1,\quad
	\tau_c(\beta'_0) = q^{-1}\beta'_0,\quad
	\tau_c(\beta'_1) = q\beta'_1.
\]
The action of $\tau_c$ on the coefficients provides an evolution direction of the $q$-Garnier system with one independent variable $\beta_0$, four dependent variables $y_1,y_3,y_5,y_7$ and seven constants $\alpha_0,\ldots,\alpha_5,\beta_0^{-1}\beta'_0$.
We do not give its explicit formula here.
Assume that
\[
	y_3 = -1,\quad
	y_4 = -\alpha_1,\quad
	y_7 = -1,\quad
	y_8 = -\alpha_3,\quad
	y_{11} = -1,\quad
	y_{12} = -\alpha_5.
\]
Note that it contains a constraint between parameters
\[
	\beta'_0 = \frac{\beta_0}{\alpha_1\alpha_3\alpha_5}.
\]
Then we have a system of $q$-difference equations called a $q$-Riccati system
\begin{align*}
	&\tau_c(y_1) = \frac{y_1\{y_1y_5y_9+\alpha_0(1-\alpha_1)y_5y_9-\alpha_0\alpha_1\alpha_2(1-\alpha_3)y_9+\alpha_0\alpha_1\alpha_2\alpha_3\alpha_4\}}{\alpha_0\alpha_1\{y_1y_5y_9+\alpha_2(1-\alpha_3)y_1y_9-\alpha_2\alpha_3\alpha_4(1-\alpha_5)y_1+\alpha_0\alpha_2\alpha_3\alpha_4\alpha_5\}}, \\
	&\tau_c(y_5) = \frac{y_5\{y_1y_5y_9+\alpha_2(1-\alpha_3)y_1y_9-\alpha_2\alpha_3\alpha_4(1-\alpha_5)y_1+\alpha_0\alpha_2\alpha_3\alpha_4\alpha_5\}}{\alpha_2\alpha_3\{y_1y_5y_9+\alpha_4(1-\alpha_5)y_1y_5-\alpha_0\alpha_4\alpha_5(1-\alpha_1)y_5+\alpha_0\alpha_1\alpha_2\alpha_4\alpha_5\}}, \\
	&\tau_c(y_9) = \frac{y_9\{y_1y_5y_9+\alpha_4(1-\alpha_5)y_1y_5-\alpha_0\alpha_4\alpha_5(1-\alpha_1)y_5+\alpha_0\alpha_1\alpha_2\alpha_4\alpha_5\}}{\alpha_4\alpha_5\{y_1y_5y_9+\alpha_0(1-\alpha_1)y_5y_9-\alpha_0\alpha_1\alpha_2(1-\alpha_3)y_9+\alpha_0\alpha_1\alpha_2\alpha_3\alpha_4\}}.
\end{align*}
We can see that the $q$-Riccati system is second order from a relation
\[
	\tau_c(y_1y_5y_9) = q^{-1}y_1y_5y_9.
\]
Recall that
\[
	y_2 = \frac{\alpha_0}{y_1},\quad
	y_6 = \frac{\alpha_2}{y_5},\quad
	y_{10} = \frac{\alpha_4}{y_9}.
\]

On the other hand, we consider a system of linear $q$-difference equations
\begin{equation}\label{eq:HGE_12}
	x(q^{-1}t) = \left(\frac{A_0}{1-t}+\frac{tA_1}{1-t}\right)x(t)
\end{equation}
with $3\times3$ matrices
\begin{align*}
	&A_0 = \begin{bmatrix}1&0&0\\1-\alpha_5&\alpha_0\alpha_5&0\\1-\alpha_5&\alpha_0\alpha_5(1-\alpha_1)&\alpha_0\alpha_1\alpha_2\alpha_5\end{bmatrix}, \\
	&A_1 = \begin{bmatrix}-\alpha_5&\alpha_0\alpha_5(1-\alpha_1)&\alpha_0\alpha_1\alpha_2\alpha_5(1-\alpha_3)\\0&-\alpha_0\alpha_1\alpha_5&\alpha_0\alpha_1\alpha_2\alpha_5(1-\alpha_3)\\0&0&-\alpha_0\alpha_1\alpha_2\alpha_3\alpha_5\end{bmatrix}.
\end{align*}
If a vector of unknown functions $x(t)={}^t[x_0(t),x_1(t),x_2(t)]$ satisfies system \eqref{eq:HGE_12}, then functions and an independent variable
\[
	y_1 = -\frac{x_0}{x_1},\quad
	y_5 = -\frac{x_1}{x_2},\quad
	y_9 = -qt\frac{x_2}{x_0},\quad
	\beta_0 = q\alpha_1\alpha_3\alpha_5t
\]
satisfy the $q$-Riccati system.
Moreover, a solution to system \eqref{eq:HGE_12} is expressed in terms of the basic hypergeometric series ${}_3\phi_2$.
The above results are summarized as the following fact.

\begin{fact}[\cite{IS,Suz2}]
The evolution direction of the $q$-Garnier system provided by $\tau_c$ admits a particular solution
\begin{align*}
	&y_1 = -\frac{x_0}{x_1},\quad
	y_2 = -\frac{\alpha_0x_1}{x_0},\quad
	y_3 = -1,\quad
	y_4 = -\alpha_1,\quad
	y_5 = -\frac{x_1}{x_2},\quad
	y_6 = -\frac{\alpha_2x_2}{x_1}, \\
	&y_7 = -1,\quad
	y_8 = -\alpha_3,\quad
	y_9 = -\frac{qtx_2}{x_0},\quad
	y_{10} = -\frac{\alpha_4x_0}{qtx_2},\quad
	y_{11} = -1,\quad
	y_{12} = -\alpha_5
\end{align*}
with
\begin{align*}
	&x_0 = {}_3\phi_2\left[\begin{array}{c}\alpha_5,\alpha_0\alpha_1\alpha_5,\alpha_0\alpha_1\alpha_2\alpha_3\alpha_5\\\alpha_0\alpha_5,\alpha_0\alpha_1\alpha_2\alpha_5\end{array};q,qt\right], \\
	&x_1 = \frac{1-\alpha_5}{1-\alpha_0\alpha_5}{}_3\phi_2\left[\begin{array}{c}q\alpha_5,\alpha_0\alpha_1\alpha_5,\alpha_0\alpha_1\alpha_2\alpha_3\alpha_5\\q\alpha_0\alpha_5,\alpha_0\alpha_1\alpha_2\alpha_5\end{array};q,qt\right], \\
	&x_2 = \frac{(1-\alpha_5)(1-\alpha_0\alpha_1\alpha_5)}{(1-\alpha_0\alpha_5)(1-\alpha_0\alpha_1\alpha_2\alpha_5)}{}_3\phi_2\left[\begin{array}{c}q\alpha_5,q\alpha_0\alpha_1\alpha_5,\alpha_0\alpha_1\alpha_2\alpha_3\alpha_5\\q\alpha_0\alpha_5,q\alpha_0\alpha_1\alpha_2\alpha_5\end{array};q,qt\right],
\end{align*}
where
\[
	{}_3\phi_2\left[\begin{array}{c}a_1,a_2,a_3\\b_1,b_2\end{array};q,t\right] = \sum_{n=0}^{\infty}\frac{(a_1;q)_n(a_2;q)_n(a_3;q)_n}{(b_1;q)_n(b_2;q)_n(q;q)_n}t^n,\quad
	(a;q)_n = \prod_{i=1}^{n}(1-aq^{i-1}).
\]
\end{fact}

%%%%%%%%%%%%%%%%%%%%%%%%%%%%%%%%%%%%%%%%%%%%%%%%%%%%%%%%%%%%%%%%
% Section 5
%%%%%%%%%%%%%%%%%%%%%%%%%%%%%%%%%%%%%%%%%%%%%%%%%%%%%%%%%%%%%%%%
\section{Degeneration of the $q$-Garnier system: from 12 vertices to 11 vertices}\label{Sec:Degeneration_12_11}

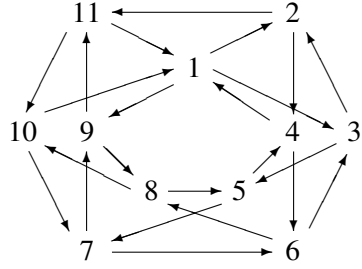
\begin{figure}[h]
	\begin{center}
	%\fbox{
	\begin{picture}(135,100)
		\put(68,69){\small$1$}\put(105,90){\small$2$}
		\put(128,45){\small$3$}\put(105,45){\small$4$}
		\put(85,22.5){\small$5$}\put(105,0){\small$6$}
		\put(52,22.5){\small$8$}\put(28,0){\small$7$}
		\put(28,45){\small$9$}\put(1,45){\small$10$}
		\put(25,90){\small$11$}
		\put(61,26){\vector(1,0){20}}%8->5
		\put(100,93.5){\vector(-1,0){60}}%2->11
		\put(40,89){\vector(2,-1){23}}%11->1
		\put(77,78){\vector(2,1){23}}%1->2
		\put(23,86){\vector(-1,-2){15}}%11->10
		\put(30.5,56){\vector(0,1){30}}%9->11
		\put(63,66){\vector(-2,-1){24}}%1->9
		\put(15,56){\vector(3,1){48}}%10->1
		\put(99,53){\vector(-3,2){21}}%4->1
		\put(108,86){\vector(0,-1){31}}%2->4
		\put(128,56){\vector(-1,2){15}}%3->2
		\put(78,72){\vector(2,-1){45}}%1->3
		\put(47,27){\vector(-2,1){32}}%8->10
		\put(9,41){\vector(1,-2){15}}%10->7
		\put(30.5,11){\vector(0,1){30}}%7->9
		\put(37,43){\vector(1,-1){11}}%9->8
		\put(124,44){\vector(-2,-1){30}}%3->5
		\put(93,33){\vector(1,1){10}}%5->4
		\put(108,41){\vector(0,-1){30}}%4->6
		\put(114,12){\vector(1,2){14}}%6->3	
		\put(81,21){\vector(-3,-1){41}}%5->7
		\put(100,8){\vector(-3,1){39}}%6->8
		\put(40,3){\vector(1,0){60}}%7->6
	\end{picture}
	%}
	\caption{Quiver $Q_{11}$}\label{Fig:deg_q-Garnier_11}
	\end{center}
\end{figure}

Since there exist $24$ arrows in the quiver $Q_{12}$, we can consider $24$ confluences from $Q_{12}$.
However, it is enough to consider two confluences $12\to1$ and $12\to2$ thanks to the symmetry of $Q_{12}$.
And they reduce to the same quiver $Q_{11}$ in Figure \ref{Fig:deg_q-Garnier_11}.
More precisely, we need the permutation $(1,2,11)(3,9,4,10)(5,8,6,7)$ after the confluence $12\to2$.
In this section, we investigate the confluence $12\to1$.

In a similar manner to Section \ref{Sec:Affine_Weyl}, we obtain simple reflections
\begin{align*}
	&r_0 = \mu_1\mu_2(2,11)\mu_2\mu_1,\quad
	r_i = \mu_{2i+1}(2i+1,2i+2)\mu_{2i+1}\quad (i=1,\ldots,4), \\
	&s_0 = \mu_1\mu_4\mu_5\mu_8(8,9)\mu_8\mu_5\mu_4\mu_1,\quad
	s_1 = \mu_2\mu_3\mu_6\mu_7\mu_{10}(10,11)\mu_{10}\mu_7\mu_6\mu_3\mu_2,
\end{align*}
simple roots
\[
	\alpha_0 = y_1y_2y_{11},\quad
	\alpha_i = y_{2i+1}y_{2i+2}\quad (i=1,\ldots,4),\quad
	\beta_0 = y_1y_4y_5y_8y_9,\quad
	\beta_1 = y_2y_3y_6y_7y_{10}y_{11}
\]
and a null root
\[
	q = \prod_{i=0}^{4}\alpha_i = \beta_0\beta_1 = \prod_{i=1}^{11}y_i
\]
from the quiver $Q_{11}$.
We omit the actions of the simple reflections on the coefficients, the simple roots and the null root here because they are similar to those given in Section \ref{Sec:Affine_Weyl}.

\begin{fact}[\cite{MOT1}]
The groups $\langle r_0,\ldots,r_4\rangle$ and $\langle s_0,s_1\rangle$ are isomorphic to the affine Weyl groups of type $A^{(1)}_4$ and $A^{(1)}_1$ respectively.
Moreover, those two groups are mutually commutative.
\end{fact}

In addition, we introduce Dynkin diagram automorphisms
\[
	\pi_1 = (1,3,5,7,9,11,2,4,6,8,10)\mu_2,\quad
	\pi_2 = \iota(2,11)(3,10)(4,9)(5,8)(6,7)
\]
and a variable
\[
	\gamma = \frac{y_2^5y_4^4y_6^3y_8^2y_{10}}{y_3y_5^2y_7^3y_9^4y_{11}^5}.
\]
Note that the definition of $\pi_1$ is suggested by that of $\pi$ in \S3 (case $A^{(1)}_5$) of \cite{BGM} and that of $\sigma_2$ in \S5.2 of \cite{MOT1}.
Also note that the definition of $\gamma$ is suggested by the kernel of the skew-symmetric matrix $\Lambda$ corresponding to $Q_{11}$ (see \cite{BGM,MOT1}).
Although the variable $\gamma$ can not be regarded as the simple root, we call $\gamma$ a simple root for the sake of convenience.
Then $\gamma$ is invariant under the action of the simple reflections and the Dynkin diagram automorphisms act on the simple roots and the null root as
\begin{align*}
	&\pi_1(\alpha_i) = \alpha_{i+1}\quad (i=0,\ldots,3),\quad
	\pi_1(\alpha_4) = \alpha_0, \\
	&\pi_1(\beta_0) = \beta_1,\quad
	\pi_1(\beta_1) = \beta_0,\quad
	\pi_1(\gamma) = q\gamma,\quad
	\pi_1(q) = q, \\
	&\pi_2(\alpha_0) = \frac{1}{\alpha_0},\quad
	\pi_2(\alpha_i) = \frac{1}{\alpha_{5-i}}\quad (i=1,\ldots,4), \\
	&\pi_2(\beta_0) = \frac{1}{\beta_0},\quad
	\pi_2(\beta_1) = \frac{1}{\beta_1},\quad
	\pi_2(\gamma) = \gamma,\quad
	\pi_2(q) = \frac1q.
\end{align*}
We omit their actions on the coefficients here.

\begin{thm}
Through the confluence $12\to1$, the simple reflections and the Dynkin diagram automorphisms are reduced as follows.
\[\begin{array}{|c||c|c|c|c|c|c|c|c|c|}\hline
	Q_{12} & r_0r_5r_0 & r_1 & r_2 & r_3 & r_4 & s_0 & s_1 & \pi_2s'_0 & \pi_3 \\\hline
	Q_{11} & r_0 & r_1 & r_2 & r_3 & r_4 & s_0 & s_1 & \pi_1^5 & \pi_2 \\\hline
\end{array}\]
The simple roots are also reduced as follows.
\[\begin{array}{|c||c|c|c|c|c|c|c|c|}\hline
	Q_{12} & \alpha_0\alpha_5 & \alpha_1 & \alpha_2 & \alpha_3 & \alpha_4 & \beta_0 & \beta_1 & \alpha_1^{-1}\alpha_2^{-2}\alpha_3^{-3}\alpha_4^{-4}\alpha_5^{-5}(\beta'_1)^5 \\\hline
	Q_{11} & \alpha_0 & \alpha_1 & \alpha_2 & \alpha_3 & \alpha_4 & \beta_0 & \beta_1 & \gamma \\\hline
\end{array}\]
\end{thm}

Note that we have not obtained the Dynkin diagram automorphism $\pi_1$ through the confluence $12\to1$ yet.

\begin{proof}
We prove the first half only for the confluence $r_0r_5r_0\to r_0$.
In the quiver $Q_{12}$, the composition $r_0r_5r_0$ acts on the coefficients as
\begin{align*}
	&r_0r_5r_0(y_1y_{12}) = \frac{Y_{1,12,2}Y_{12,2,11}}{y_2y_{12}Y_{2,11,1}Y_{11,1,12}},\quad
	r_0r_5r_0(y_2) = \frac{Y_{2,11,1}}{y_{11}Y_{1,12,2}}, \\
	&r_0r_5r_0(y_3) = \frac{y_1y_3y_{11}Y_{12,2,11}}{Y_{11,1,12}},\quad
	r_0r_5r_0(y_4) = \frac{y_2y_4y_{12}Y_{11,1,12}}{Y_{12,2,11}},\quad
	r_0r_5r_0(y_i) = y_i\quad (i=5,\ldots,8), \\
	&r_0r_5r_0(y_9) = \frac{y_1y_9y_{12}Y_{2,11,1}}{Y_{1,12,2}},\quad
	r_0r_5r_0(y_{10}) = \frac{y_2y_{10}y_{11}Y_{1,12,2}}{Y_{2,11,1}},\quad
	r_0r_5r_0(y_{11}) = \frac{Y_{11,1,12}}{y_1Y_{12,2,11}},
\end{align*}
where $Y_{i,j,k}=1+y_i+y_iy_j+y_iy_jy_k$.
Replacing $y_{12}\to\varepsilon^{-1}y_1,\ y_1\to\varepsilon$, taking a limit $\varepsilon\to0$ and replacing $r_0r_5r_0\to r_0$, we obtain
\begin{align*}
	&r_0(y_1) = \frac{Y_{1,2}}{y_2Y_{11,1}},\quad
	r_0(y_2) = \frac{Y_{2,11}}{y_{11}Y_{1,2}},\quad
	r_0(y_3) = \frac{y_1y_3y_{11}Y_{2,11}}{Y_{11,1}},\quad
	r_0(y_4) = \frac{y_2y_4Y_{11,1}}{Y_{2,11}}, \\
	&r_0(y_i) = y_i\quad (i=5,\ldots,8),\quad
	r_0(y_9) = \frac{y_1y_9Y_{2,11}}{Y_{1,2}},\quad
	r_0(y_{10}) = \frac{y_2y_{10}y_{11}Y_{1,2}}{Y_{2,11}},\quad
	r_0(y_{11}) = \frac{Y_{11,1}}{y_1Y_{2,11}},
\end{align*}
where $Y_{i,j}=1+y_i+y_iy_j$, which is the same as the action of $r_0$ in $Q_{11}$.
Note that this confluence is interpreted at the level of the mutations and the permutations as
\begin{align*}
	r_0r_5r_0 &= \mu_2(1,2)\mu_2\mu_{12}(11,12)\mu_{12}\mu_2(1,2)\mu_2 \\
	&= \mu_2\mu_1\mu_{12}(11,12)\mu_{12}\mu_1\mu_2 \\
	&\xrightarrow{12\to1} \mu_2\mu_1(11,1)\mu_1\mu_2 \\
	&= r_0.
\end{align*}
Also note that $\mu_1(1,2)\mu_1=\mu_2(1,2)\mu_2$ and $\mu_1\mu_2(2,11)\mu_2\mu_1=\mu_2\mu_1(11,1)\mu_1\mu_2$ hold.
We can prove for the other confluences in a similar manner, whose details we omit here.

The latter half is shown as
\begin{align*}
	&\left(\alpha_0\alpha_5,\alpha_1,\ldots,\alpha_4,\beta_0,\beta_1,\alpha_1^{-1}\alpha_2^{-2}\alpha_3^{-3}\alpha_4^{-4}\alpha_5^{-5}(\beta'_1)^5\right) \\
	&= \left(y_1y_2y_{11}y_{12},y_3y_4,\ldots,y_9y_{10},y_1y_4y_5y_8y_9y_{12},y_2y_3y_6y_7y_{10}y_{11},\frac{y_2^5y_4^4y_6^3y_8^2y_{10}}{y_3y_5^2y_7^3y_9^4y_{11}^5}\right) \\
	&\xrightarrow{12\to1} \left(y_1y_2y_{11},y_3y_4,\ldots,y_9y_{10},y_1y_4y_5y_8y_9,y_2y_3y_6y_7y_{10}y_{11},\frac{y_2^5y_4^4y_6^3y_8^2y_{10}}{y_3y_5^2y_7^3y_9^4y_{11}^5}\right) \\
	&= \left(\alpha_0,\alpha_1,\ldots,\alpha_4,\beta_0,\beta_1,\gamma\right).
\end{align*}
\end{proof}

The translation subgroup of $\langle r_0,\ldots,r_4,s_0,s_1,\pi_1,\pi_2\rangle$ is generated by
\begin{align*}
	&T_i = r_ir_{i+1}r_{i+2}r_{i+3}r_{i+4}r_{i+3}r_{i+2}r_{i+1}\quad (i=0,\ldots,4), \\
	&U_0 = s_0s_1,\quad
	U_1 = s_1s_0,\quad
	V = \pi_1r_4r_3r_2r_1s_1,\quad
	V' = \pi_1^5s_1,
\end{align*}
where the indices of $r_i$ are congruent modulo $5$.
They act on the simple roots as
\begin{align*}
	&T_i(\alpha_j) = q^{-\delta_{j,i-1}+2\delta_{j,i}-\delta_{j,i+1}}\alpha_j,\quad
	T_i(\beta_l) = \beta_l,\quad
	T_i(\gamma) = \gamma, \\
	&U_k(\alpha_j) = \alpha_j,\quad
	U_k(\beta_l) = q^{4\delta_{l,k}-2}\beta_l,\quad
	U_k(\gamma) = \gamma, \\
	&V(\alpha_j) = q^{\delta_{j,0}-\delta_{j,1}}\alpha_j,\quad
	V(\beta_l) = q^{\delta_{l,0}-\delta_{l,1}}\beta_l,\quad
	V(\gamma) = q\gamma, \\
	&V'(\alpha_j) = \alpha_j,\quad
	V'(\beta_l) = q^{\delta_{l,0}-\delta_{l,1}}\beta_l,\quad
	V'(\gamma) = q^5\gamma
\end{align*}
for $i,j=0,\ldots,4$ and $k,l=0,1$, where $\delta_{i,j}$ is the Kronecker delta whose indices are congruent modulo $5$.

\begin{cor}
Through the confluence $12\to1$, the translations are reduced as follows.
\[\begin{array}{|c||c|c|c|c|c|c|c|c|c|}\hline
	Q_{12} & T_5T_0 & T_1 & T_2 & T_3 & T_4 & U_0 & U_1 & V & (V')^{-1}V \\\hline
	Q_{11} & T_0 & T_1 & T_2 & T_3 & T_4 & U_0 & U_1 & V & V' \\\hline
\end{array}\]
\end{cor}

Note that the translation $\tau_c$ in $Q_{12}$ is reduced to that $\tau_c=V'U_1=\pi_1^5s_0$ in $Q_{11}$ through the confluence $12\to1$.

%%%%%%%%%%%%%%%%%%%%%%%%%%%%%%%%%%%%%%%%%%%%%%%%%%%%%%%%%%%%%%%%
% Section 6
%%%%%%%%%%%%%%%%%%%%%%%%%%%%%%%%%%%%%%%%%%%%%%%%%%%%%%%%%%%%%%%%
\section{Degeneration of the $q$-Garnier system: from 11 vertices to 10 vertices}\label{Sec:Degeneration_11_10}

Since there exist $23$ arrows in the quiver $Q_{11}$, we can consider $23$ confluences from $Q_{11}$.
Then the obtained quivers are reduced to five quivers in Figure \ref{Fig:deg_q-Garnier_101}, \ref{Fig:deg_q-Garnier_102}, \ref{Fig:deg_q-Garnier_103}, \ref{Fig:deg_q-Garnier_104} and \ref{Fig:deg_q-Garnier_105} via mutations and permutations.
Since those quivers are obtained through the confluences
\[
	4 \to 5,\quad
	6 \to 4,\quad
	5 \to 8,\quad
	11 \to 2,\quad
	1 \to 11,
\]
we investigate them in this section.
Note that we have not clarified whether those five quivers are truly different yet.

%%%%%%%%%%%%%%%%%%%%%%%%%%%%%%%%%%%%%%%%%%%%%%%%%%%%%%%%%%%%%%%%
\subsection{Confluence $4\to5$}\label{Sec:Degeneration_11_101}

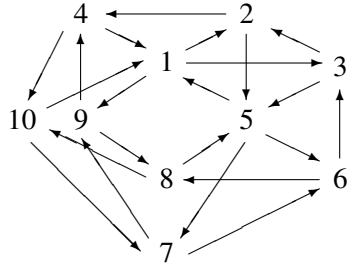
\begin{figure}[h]
	\begin{center}
	%\fbox{
	\begin{picture}(135,100)
		\put(60,72){\small$1$}\put(90,90){\small$2$}
		\put(125,70){\small$3$}\put(90,50){\small$5$}
		\put(125,28){\small$6$}\put(28,90){\small$4$}
		\put(60,0){\small$7$}\put(60,28){\small$8$}
		\put(28,50){\small$9$}\put(3,50){\small$10$}
		\put(30.5,61){\vector(0,1){25}}%9-4
		\put(24,87){\vector(-1,-2){13}}%4-10
		\put(18,58){\vector(2,1){36}}%10-1
		\put(55,32){\vector(-2,1){36}}%10-8
		\put(85,93.5){\vector(-1,0){45}}%2-4
		\put(70,80){\vector(2,1){15}}%1-2
		\put(86,61){\vector(-2,1){17}}%1-5
		\put(55,70){\vector(-3,-2){18}}%1-9
		\put(120,79){\vector(-2,1){18}}%2-3
		\put(120,68){\vector(-2,-1){19}}%3-5
		\put(100,48){\vector(2,-1){20}}%6-5
		\put(40,88){\vector(2,-1){15}}%4-1
		\put(92.5,86){\vector(0,-1){25}}%2-5
		\put(127.5,39){\vector(0,1){26}}%6-3
		\put(70,75){\vector(1,0){51}}%1-3
		\put(120,31){\vector(-1,0){51}}%6-8
		\put(71,3){\vector(2,1){49}}%6-7
		\put(12,45){\vector(1,-1){41}}%10-7
		\put(56,9){\vector(-2,3){25}}%9-7
		\put(92,45){\vector(-2,-3){24}}%7-5
		\put(38,49){\vector(3,-2){18}}%9-8
		\put(69,37){\vector(3,2){17}}%8-5
	\end{picture}
	%}
	\caption{Quiver $Q_{101}$}\label{Fig:deg_q-Garnier_101}
	\end{center}
\end{figure}

We change the vertex $11$ to $4$ after the confluence.
In a similar manner to Section \ref{Sec:Affine_Weyl}, we obtain simple reflections
\begin{align*}
	&r_0 = \mu_1\mu_2(2,4)\mu_2\mu_1,\quad
	r_1 = \mu_3\mu_5(5,6)\mu_5\mu_3,\quad
	r_2 = \mu_7(7,8)\mu_7,\quad
	r_3 = \mu_9(9,10)\mu_9, \\
	&s_0 = \mu_1\mu_5\mu_8(8,9)\mu_8\mu_5\mu_1,\quad
	s_1 = \mu_2\mu_3\mu_6\mu_7\mu_{10}(10,4)\mu_{10}\mu_7\mu_6\mu_3\mu_2,
\end{align*}
simple roots
\[
	\alpha_0 = y_1y_2y_4,\quad
	\alpha_1 = y_3y_5y_6,\quad
	\alpha_2 = y_7y_8,\quad
	\alpha_3 = y_9y_{10},\quad
	\beta_0 = y_1y_5y_8y_9,\quad
	\beta_1 = y_2y_3y_4y_6y_7y_{10}
\]
and a null root
\[
	q = \prod_{i=0}^{3}\alpha_i = \beta_0\beta_1 = \prod_{i=1}^{10}y_i
\]
from the quiver $Q_{101}$.
We omit the actions of the simple reflections on the coefficients, the simple roots and the null root here because they are similar to those given in Section \ref{Sec:Affine_Weyl}.

\begin{fact}[\cite{MOT1}]
The groups $\langle r_0,\ldots,r_3\rangle$ and $\langle s_0,s_1\rangle$ are isomorphic to the affine Weyl groups of type $A^{(1)}_3$ and $A^{(1)}_1$ respectively.
Moreover, those two groups are mutually commutative.
\end{fact}

In addition, we introduce Dynkin diagram automorphisms
\begin{align*}
	&\pi_1 = (1,3,6,8,10)(2,5,7,9,4)\mu_2\mu_6, \\
	&\pi_2 = (1,7)(2,8)(3,9)(4,6)(5,10)\mu_4\mu_6, \\
	&\pi_3 = \iota(2,4)(3,10)(5,9)(6,8)\mu_8\mu_6
\end{align*}
and a variable
\[
	\gamma = \frac{y_2^2y_5y_6^3y_8^2y_{10}}{y_3y_4^2y_9}.
\]
Similarly to Section \ref{Sec:Degeneration_12_11}, we call $\gamma$ a simple root for the sake of convenience.
Then $\gamma$ is invariant under the action of the simple reflections and the Dynkin diagram automorphisms act on the simple roots and the null root as
\begin{align*}
	&\pi_1(\alpha_i) = \alpha_{i+1}\quad (i=0,\ldots,2),\quad
	\pi_1(\alpha_3) = \alpha_0, \\
	&\pi_1(\beta_0) = \beta_1,\quad
	\pi_1(\beta_1) = \beta_0,\quad
	\pi_1(\gamma) = q\gamma,\quad
	\pi_1(q) = q, \\
	&\pi_2(\alpha_0) = \alpha_2,\quad
	\pi_2(\alpha_1) = \alpha_3,\quad
	\pi_2(\alpha_2) = \alpha_0,\quad
	\pi_2(\alpha_3) = \alpha_1, \\
	&\pi_2(\beta_0) = \beta_1,\quad
	\pi_2(\beta_1) = \beta_0,\quad
	\pi_2(\gamma) = \gamma,\quad
	\pi_2(q) = q, \\
	&\pi_3(\alpha_0) = \frac{1}{\alpha_0},\quad
	\pi_3(\alpha_i) = \frac{1}{\alpha_{4-i}}\quad (i=1,2,3), \\
	&\pi_3(\beta_0) = \frac{1}{\beta_0},\quad
	\pi_3(\beta_1) = \frac{1}{\beta_1},\quad
	\pi_3(\gamma) = \gamma,\quad
	\pi_3(q) = \frac1q.
\end{align*}
We omit their actions on the coefficients here.

\begin{thm}
Through the confluence $4\to5$, the simple reflections and the Dynkin diagram automorphisms are reduced as follows.
\[\begin{array}{|c||c|c|c|c|c|c|c|c|c|}\hline
	Q_{11} & r_0 & r_1r_2r_1 & r_3 & r_4 & s_0 & s_1 & r_2\pi_1 & r_1r_0\pi_1^3 & r_2r_3\pi_2 \\\hline
	Q_{101} & r_0 & r_1 & r_2 & r_3 & s_0 & s_1 & \pi_1 & \pi_2 & \pi_3 \\\hline
\end{array}\]
The simple roots are also reduced as follows.
\[\begin{array}{|c||c|c|c|c|c|c|c|}\hline
	Q_{11} & \alpha_0 & \alpha_1\alpha_2 & \alpha_3 & \alpha_4 & \beta_0 & \beta_1 & \alpha_1^{-\frac35}\alpha_2^{\frac95}\alpha_3^{\frac65}\alpha_4^{\frac35}\gamma^{\frac25} \\\hline
	Q_{101} & \alpha_0 & \alpha_1 & \alpha_2 & \alpha_3 & \beta_0 & \beta_1 & \gamma \\\hline
\end{array}\]
\end{thm}

We can prove this theorem in a similar manner to Section \ref{Sec:Degeneration_12_11}, whose detail we omit here.
Note that we replace the coefficient $y_{11}$ with $y_4$ after the confluence procedure.

The translation subgroup of $\langle r_0,\ldots,r_3,s_0,s_1,\pi_1,\pi_2,\pi_3\rangle$ is generated by
\begin{align*}
	&T_i = r_ir_{i+1}r_{i+2}r_{i+3}r_{i+2}r_{i+1}\quad (i=0,\ldots,3), \\
	&U_0 = s_0s_1,\quad
	U_1 = s_1s_0,\quad
	V = \pi_1r_3r_2r_1s_1,\quad
	V' = \pi_2\pi_1^2s_1,
\end{align*}
where the indices of $r_i$ are congruent modulo $4$.
They act on the simple roots as
\begin{align*}
	&T_i(\alpha_j) = q^{-\delta_{j,i-1}+2\delta_{j,i}-\delta_{j,i+1}}\alpha_j,\quad
	T_i(\beta_l) = \beta_l,\quad
	T_i(\gamma) = \gamma, \\
	&U_k(\alpha_j) = \alpha_j,\quad
	U_k(\beta_l) = q^{4\delta_{l,k}-2}\beta_l,\quad
	U_k(\gamma) = \gamma, \\
	&V(\alpha_j) = q^{\delta_{j,0}-\delta_{j,1}}\alpha_j,\quad
	V(\beta_l) = q^{\delta_{l,0}-\delta_{l,1}}\beta_l,\quad
	V(\gamma) = q\gamma, \\
	&V'(\alpha_j) = \alpha_j,\quad
	V'(\beta_l) = q^{\delta_{l,0}-\delta_{l,1}}\beta_l,\quad
	V'(\gamma) = q^2\gamma
\end{align*}
for $i,j=0,\ldots,3$ and $k,l=0,1$, where $\delta_{i,j}$ is the Kronecker delta whose indices are congruent modulo $4$.

\begin{cor}
Through the confluence $4\to5$, the translations are reduced as follows.
\[\begin{array}{|c||c|c|c|c|c|c|c|c|}\hline
	Q_{11} & T_0 & T_1T_2 & T_3 & T_4 & U_0 & U_1 & V & V' \\\hline
	Q_{101} & T_0 & T_1 & T_2 & T_3 & U_0 & U_1 & V & V' \\\hline
\end{array}\]
\end{cor}

Note that the translation $\tau_c$ in $Q_{11}$ is reduced to that $\tau_c=V'U_1=\pi_2\pi_1^2s_0$ in $Q_{101}$ through the confluence $4\to5$.

%%%%%%%%%%%%%%%%%%%%%%%%%%%%%%%%%%%%%%%%%%%%%%%%%%%%%%%%%%%%%%%%
\subsection{Confluence $6\to4$}\label{Sec:Degeneration_11_102}

\begin{figure}[h]
	\begin{center}
	%\fbox{
	\begin{picture}(135,100)
		\put(60,72){\small$1$}\put(90,90){\small$2$}
		\put(125,70){\small$3$}\put(90,50){\small$4$}
		\put(125,28){\small$5$}\put(28,90){\small$6$}
		\put(60,0){\small$7$}\put(60,28){\small$8$}
		\put(28,50){\small$9$}\put(3,50){\small$10$}
		\put(30.5,61){\vector(0,1){25}}%9-4
		\put(24,87){\vector(-1,-2){13}}%4-10
		\put(18,58){\vector(2,1){36}}%10-1
		\put(55,32){\vector(-2,1){36}}%10-8
		\put(85,93.5){\vector(-1,0){45}}%2-4
		\put(70,80){\vector(2,1){15}}%1-2
		\put(86,61){\vector(-2,1){17}}%1-5
		\put(55,70){\vector(-3,-2){18}}%1-9
		\put(120,79){\vector(-2,1){18}}%2-3
		\put(100,59){\vector(2,1){20}}%3-4
		\put(120,38){\vector(-2,1){20}}%4-5
		\put(40,88){\vector(2,-1){15}}%4-1
		\put(92.5,86){\vector(0,-1){25}}%2-5
		\put(127.5,65){\vector(0,-1){26}}%5-3
		\put(70,75){\vector(1,0){51}}%1-3
		\put(70,31){\vector(1,0){51}}%6-8
		\put(120,28){\vector(-2,-1){49}}%5-7
		\put(12,45){\vector(1,-1){41}}%10-7
		\put(56,9){\vector(-2,3){25}}%9-7
		\put(68,9){\vector(2,3){24}}%4-7
		\put(38,49){\vector(3,-2){18}}%9-8
		\put(86,48){\vector(-3,-2){17}}%4-8
	\end{picture}
	%}
	\caption{Quiver $Q_{102}$}\label{Fig:deg_q-Garnier_102}
	\end{center}
\end{figure}
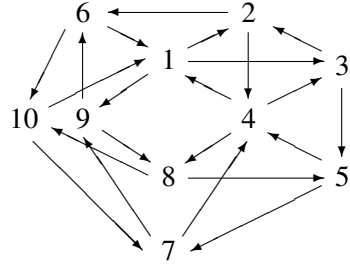

We change the vertex $11$ to $6$ after the confluence.
In a similar manner to Section \ref{Sec:Affine_Weyl}, we obtain simple reflections
\[
	r_0 = \mu_1\mu_2(2,6)\mu_2\mu_1,\quad
	r_1 = \mu_3\mu_4(4,5)\mu_4\mu_3,\quad
	r_2 = \mu_7(7,8)\mu_7,\quad
	r_3 = \mu_9(9,10)\mu_9,
\]
simple roots
\[
	\alpha_0 = y_1y_2y_6,\quad
	\alpha_1 = y_3y_4y_5,\quad
	\alpha_2 = y_7y_8,\quad
	\alpha_3 = y_9y_{10}
\]
and a null root
\[
	q = \prod_{i=0}^{3}\alpha_i = \prod_{i=1}^{10}y_i
\]
from the quiver $Q_{102}$.
We omit the actions of the simple reflections on the coefficients, the simple roots and the null root here because they are similar to those given in Section \ref{Sec:Affine_Weyl}.

\begin{fact}[\cite{MOT1}]
The group $\langle r_0,\ldots,r_3\rangle$ is isomorphic to the affine Weyl group of type $A^{(1)}_3$.
\end{fact}

In addition, we introduce Dynkin diagram automorphisms
\begin{align*}
	&\pi_1 = \iota(1,3)(4,6)(5,9)(7,10)\mu_9\mu_8\mu_2\mu_5, \\
	&\pi_2 = \iota(1,7)(2,4)(3,5)(6,8)(9,10)\mu_4\mu_2, \\
	&\pi_3 = (1,7)(2,5)(3,4)(6,8)\mu_5\mu_2
\end{align*}
and variables
\[
	\gamma_1 = \frac{y_2^2y_3y_4y_{10}}{y_5y_9},\quad
	\gamma_2 = \frac{y_1y_2y_4^2y_5^2y_8^3}{y_3^2y_6^3y_7}.
\]
Similarly to Section \ref{Sec:Degeneration_12_11}, we call $\gamma_1,\gamma_2$ simple roots for the sake of convenience.
Then $\gamma_1,\gamma_2$ are invariant under the action of the simple reflections and the Dynkin diagram automorphisms act on the simple roots and the null root as
\begin{align*}
	&\pi_1(\alpha_0) = \frac{1}{\alpha_1},\quad
	\pi_1(\alpha_1) = \frac{1}{\alpha_0},\quad
	\pi_1(\alpha_2) = \frac{1}{\alpha_3},\quad
	\pi_1(\alpha_3) = \frac{1}{\alpha_2}, \\
	&\pi_1(\gamma_1) = q^{-1}\gamma_1,\quad
	\pi_1(\gamma_2) = q\gamma_2,\quad
	\pi_1(q) = \frac1q, \\
	&\pi_2(\alpha_i) = \frac{1}{\alpha_{2-i}}\quad (i=0,1,2),\quad
	\pi_2(\alpha_3) = \frac{1}{\alpha_3}, \\
	&\pi_2(\gamma_1) = \gamma_1,\quad
	\pi_2(\gamma_2) = \gamma_2,\quad
	\pi_2(q) = \frac1q, \\
	&\pi_3(\alpha_i) = \alpha_{2-i}\quad (i=0,1,2),\quad
	\pi_3(\alpha_3) = \alpha_3, \\
	&\pi_3(\gamma_1) = \gamma_1,\quad
	\pi_3(\gamma_2) = \frac{1}{\gamma_2},\quad
	\pi_3(q) = q.
\end{align*}
We omit their actions on the coefficients here.

\begin{thm}
Through the confluence $6\to4$, the simple reflections and the Dynkin diagram automorphisms are reduced as follows.
\[\begin{array}{|c||c|c|c|c|c|c|c|}\hline
	Q_{11} & r_0 & r_1r_2r_1 & r_3 & r_4 & r_2r_3r_4\pi_1\pi_2 & r_2r_3r_4r_0s_1\pi_1^2\pi_2 & r_2\pi_2\pi_1 \\\hline
	Q_{102} & r_0 & r_1 & r_2 & r_3 & \pi_1 & \pi_2 & \pi_3\pi_1\pi_3 \\\hline
\end{array}\]
The simple roots are also reduced as follows.
\[\begin{array}{|c||c|c|c|c|c|c|}\hline
	Q_{11} & \alpha_0 & \alpha_1\alpha_2 & \alpha_3 & \alpha_4 & \alpha_1^{\frac15}\alpha_2^{-\frac35}\alpha_3^{-\frac25}\alpha_4^{-\frac15}\beta_1\gamma^{\frac15} & \alpha_1^{-\frac35}\alpha_2^{\frac95}\alpha_3^{\frac65}\alpha_4^{\frac35}\beta_0\beta_1^{-1}\gamma^{\frac25} \\\hline
	Q_{102} & \alpha_0 & \alpha_1 & \alpha_2 & \alpha_3 & \gamma_1 & \gamma_2 \\\hline
\end{array}\]
\end{thm}

We can prove this theorem in a similar manner to Section \ref{Sec:Degeneration_12_11}, whose detail we omit here.
Note that we replace the coefficients $y_{11}$ with $y_6$ after the confluence procedure.
Note that we have not obtained the Dynkin diagram automorphism $\pi_3$ through the confluence $6\to4$ yet.

The translation subgroup of $\langle r_0,\ldots,r_3,\pi_1,\pi_2,\pi_3\rangle$ is generated by
\[
	T_i = r_ir_{i+1}r_{i+2}r_{i+3}r_{i+2}r_{i+1}\quad (i=0,\ldots,3),\quad
	U = \pi_2\pi_1r_3r_2r_1,\quad
	V = (\pi_3\pi_2\pi_1)^2,\quad
	V' = (\pi_1\pi_3)^4,
\]
where the indices of $r_i$ are congruent modulo $4$.
They act on the simple roots as
\begin{align*}
	&T_i(\alpha_j) = q^{-\delta_{j,i-1}+2\delta_{j,i}-\delta_{j,i+1}}\alpha_j,\quad
	T_i(\gamma_1) = \gamma_1,\quad
	T_i(\gamma_2) = \gamma_2, \\
	&U(\alpha_j) = q^{\delta_{j,0}-\delta_{j,1}}\alpha_j,\quad
	U(\gamma_1) = q\gamma_1,\quad
	U(\gamma_2) = q^{-1}\gamma_2, \\
	&V(\alpha_j) = \alpha_j,\quad
	V(\gamma_1) = q^2\gamma_1,\quad
	V(\gamma_2) = \gamma_2, \\
	&V'(\alpha_j) = \alpha_j,\quad
	V'(\gamma_1) = \gamma_1,\quad
	V'(\gamma_2) = q^4\gamma_2
\end{align*}
for $i,j=0,\ldots,3$, where $\delta_{i,j}$ is the Kronecker delta whose indices are congruent modulo $4$.

\begin{cor}
Through the confluence $6\to4$, the translations are reduced as follows.
\[\begin{array}{|c||c|c|c|c|c|c|c|}\hline
	Q_{11} & T_0 & T_1T_2 & T_3 & T_4 & VU_1 & V'U_1 & V' \\\hline
	Q_{102} & T_0 & T_1 & T_2 & T_3 & U & V & V' \\\hline
\end{array}\]
\end{cor}

Note that the translation $\tau_c$ in $Q_{11}$ is reduced to that $\tau_c=V=\pi_3\pi_2\pi_1\pi_3\pi_2\pi_1$ in $Q_{102}$ through the confluence $6\to4$.

%%%%%%%%%%%%%%%%%%%%%%%%%%%%%%%%%%%%%%%%%%%%%%%%%%%%%%%%%%%%%%%%
\subsection{Confluence $5\to8$}

\begin{figure}[h]
	\begin{center}
	%\fbox{
	\begin{picture}(100,135)
		\put(68,69){\small$1$}\put(105,90){\small$2$}
		\put(128,45){\small$3$}\put(105,45){\small$4$}
		\put(28,90){\small$5$}\put(105,0){\small$6$}
		\put(28,0){\small$7$}\put(68,21){\small$8$}
		\put(28,45){\small$9$}\put(1,45){\small$10$}
		\put(100,93.5){\vector(-1,0){60}}%2->5
		\put(40,89){\vector(2,-1){23}}%5->1
		\put(77,78){\vector(2,1){23}}%1->2
		\put(23,86){\vector(-1,-2){15}}%5->10
		\put(30.5,56){\vector(0,1){30}}%9->5
		\put(63,66){\vector(-2,-1){24}}%1->9
		\put(15,56){\vector(3,1){48}}%10->1
		\put(99,53){\vector(-3,2){21}}%4->1
		\put(108,86){\vector(0,-1){31}}%2->4
		\put(128,56){\vector(-1,2){15}}%3->2
		\put(78,72){\vector(2,-1){45}}%1->3
		\put(62,24){\vector(-2,1){46}}%8->10
		\put(9,41){\vector(1,-2){15}}%10->7
		\put(30.5,11){\vector(0,1){30}}%7->9
		\put(39,43){\vector(2,-1){23}}%9->8
		\put(124,45){\vector(-2,-1){43}}%3->8
		\put(80,30){\vector(3,2){20}}%4-8
		\put(108,41){\vector(0,-1){30}}%4->6
		\put(114,12){\vector(1,2){14}}%6->3	
		\put(62,19){\vector(-2,-1){22}}%7-8
		\put(102,9){\vector(-2,1){22}}%6->5
		\put(40,3){\vector(1,0){60}}%7->6
	\end{picture}
	%}
	\caption{Quiver $Q_{103}$}\label{Fig:deg_q-Garnier_103}
	\end{center}
\end{figure}
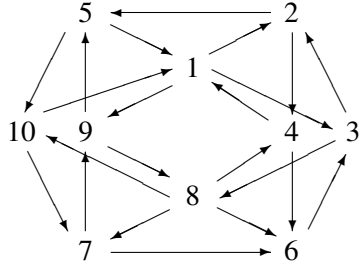

We change the vertex $11$ to $5$ after the confluence.
In a similar manner to Section \ref{Sec:Affine_Weyl}, we obtain simple reflections
\begin{align*}
	&r_0 = \mu_1\mu_2(2,5)\mu_2\mu_1,\quad
	r_1 = \mu_3(3,4)\mu_3,\quad
	r_2 = \mu_6\mu_7(7,8)\mu_7\mu_6,\quad
	r_3 = \mu_9(9,10)\mu_9, \\
	&s_0 = \mu_1\mu_4\mu_8(8,9)\mu_8\mu_4\mu_1,\quad
	s_1 = \mu_2\mu_3\mu_6\mu_7\mu_{10}(10,5)\mu_{10}\mu_7\mu_6\mu_3\mu_2,
\end{align*}
simple roots
\[
	\alpha_0 = y_1y_2y_5,\quad
	\alpha_1 = y_3y_4,\quad
	\alpha_3 = y_6y_7y_8,\quad
	\alpha_5 = y_9y_{10},\quad
	\beta_0 = y_1y_4y_8y_9,\quad
	\beta_1 = y_2y_3y_5y_6y_7y_{10}
\]
and a null root
\[
	q = \prod_{i=0}^{3}\alpha_i = \beta_0\beta_1 = \prod_{i=1}^{10}y_i
\]
from the quiver $Q_{103}$.
We omit the actions of the simple reflections on the coefficients, the simple roots and the null root here because they are similar to those given in Section \ref{Sec:Affine_Weyl}.

\begin{fact}[\cite{MOT1}]
The groups $\langle r_0,\ldots,r_3\rangle$ and $\langle s_0,s_1\rangle$ are isomorphic to the affine Weyl groups of type $A^{(1)}_3$ and $A^{(1)}_1$ respectively.
Moreover, those two groups are mutually commutative.
\end{fact}

In addition, we introduce Dynkin diagram automorphisms
\begin{align*}
	&\pi_1 = (1,3,8,10)(2,4,6,7,9,5)\mu_7\mu_2, \\
	&\pi_2 = \iota(2,5)(3,10)(4,9)(6,7), \\
	&\pi_3 = (1,8)(2,7)(3,10)(4,9)(5,6)
\end{align*}
and a variable
\[
	\gamma = \frac{y_2y_4y_6}{y_5y_7y_9}.
\]
Similarly to Section \ref{Sec:Degeneration_12_11}, we call $\gamma$ a simple root for the sake of convenience.
Then $\gamma$ is invariant under the action of the simple reflections and the Dynkin diagram automorphisms act on the simple roots and the null root as
\begin{align*}
	&\pi_1(\alpha_i) = \alpha_{i+1}\quad (i=0,\ldots,2),\quad
	\pi_1(\alpha_3) = \alpha_0, \\
	&\pi_1(\beta_0) = \beta_1,\quad
	\pi_1(\beta_1) = \beta_0,\quad
	\pi_1(\gamma) = \gamma,\quad
	\pi_1(q) = q, \\
	&\pi_2(\alpha_0) = \frac{1}{\alpha_0},\quad
	\pi_2(\alpha_i) = \frac{1}{\alpha_{4-i}}\quad (i=1,2,3), \\
	&\pi_2(\beta_0) = \frac{1}{\beta_0},\quad
	\pi_2(\beta_1) = \frac{1}{\beta_1},\quad
	\pi_2(\gamma) = \gamma,\quad
	\pi_2(q) = \frac1q, \\
	&\pi_3(\alpha_0) = \alpha_2,\quad
	\pi_3(\alpha_1) = \alpha_3,\quad
	\pi_3(\alpha_2) = \alpha_0,\quad
	\pi_3(\alpha_3) = \alpha_1, \\
	&\pi_3(\beta_0) = \beta_0,\quad
	\pi_3(\beta_1) = \beta_1,\quad
	\pi_3(\gamma) = \frac{1}{\gamma},\quad
	\pi_3(q) = q.
\end{align*}
We omit their actions on the coefficients here.

\begin{thm}
Through the confluence $5\to8$, the simple reflections and the Dynkin diagram automorphisms are reduced as follows.
\[\begin{array}{|c||c|c|c|c|c|c|c|c|}\hline
	Q_{11} & r_0 & r_1 & r_2r_3r_2 & r_4 & s_0 & s_1 & r_3\pi_1 & \pi_2 \\\hline
	Q_{103} & r_0 & r_1 & r_2 & r_3 & s_0 & s_1 & \pi_1 & \pi_2 \\\hline
\end{array}\]
The simple roots are also reduced as follows.
\[\begin{array}{|c||c|c|c|c|c|c|c|}\hline
	Q_{11} & \alpha_0 & \alpha_1 & \alpha_2\alpha_3 & \alpha_4 & \beta_0 & \beta_1 & \alpha_1^{\frac15}\alpha_2^{\frac25}\alpha_3^{-\frac25}\alpha_4^{-\frac15}\gamma^{\frac15} \\\hline
	Q_{103} & \alpha_0 & \alpha_1 & \alpha_2 & \alpha_3 & \beta_0 & \beta_1 & \gamma \\\hline
\end{array}\]
\end{thm}

We can prove this theorem in a similar manner to Section \ref{Sec:Degeneration_12_11}, whose detail we omit here.
Note that we replace the coefficients $y_{11}$ with $y_5$ after the confluence procedure.
Also note that we have not obtained the Dynkin diagram automorphism $\pi_3$ through the confluence $5\to8$ yet.

The translation subgroup of $\langle r_0,\ldots,r_3,s_0,s_1,\pi_1,\pi_2,\pi_3\rangle$ is generated by
\begin{align*}
	&T_i = r_ir_{i+1}r_{i+2}r_{i+3}r_{i+2}r_{i+1}\quad (i=0,\ldots,3), \\
	&U_0 = s_0s_1,\quad
	U_1 = s_1s_0,\quad
	V = \pi_1r_3r_2r_1s_1,
\end{align*}
where the indices of $r_i$ are congruent modulo $4$.
They act on the simple roots as
\begin{align*}
	&T_i(\alpha_j) = q^{-\delta_{j,i-1}+2\delta_{j,i}-\delta_{j,i+1}}\alpha_j,\quad
	T_i(\beta_l) = \beta_l,\quad
	T_i(\gamma) = \gamma, \\
	&U_k(\alpha_j) = \alpha_j,\quad
	U_k(\beta_l) = q^{4\delta_{l,k}-2}\beta_l,\quad
	U_k(\gamma) = \gamma, \\
	&V(\alpha_j) = q^{\delta_{j,0}-\delta_{j,1}}\alpha_j,\quad
	V(\beta_l) = q^{\delta_{l,0}-\delta_{l,1}}\beta_l,\quad
	V(\gamma) = q\gamma
\end{align*}
for $i,j=0,\ldots,3$ and $k,l=0,1$, where $\delta_{i,j}$ is the Kronecker delta whose indices are congruent modulo $4$.

\begin{cor}
Through the confluence $5\to8$, the translations are reduced as follows.
\[\begin{array}{|c||c|c|c|c|c|c|c|}\hline
	Q_{11} & T_0 & T_1 & T_2T_3 & T_4 & U_0 & U_1 & V \\\hline
	Q_{103} & T_0 & T_1 & T_2 & T_3 & U_0 & U_1 & V \\\hline
\end{array}\]
\end{cor}

\begin{rem}
The action of $\tau_c$ in $Q_{11}$ on the coefficients is not convergent through the confluence $5\to8$.
Besides, the coefficients are not invariant under the action of $\pi_1^4$, although the simple roots are invariant.
We have not clarified the reasons for those phenomena.
\end{rem}

%%%%%%%%%%%%%%%%%%%%%%%%%%%%%%%%%%%%%%%%%%%%%%%%%%%%%%%%%%%%%%%%
\subsection{Confluence $11\to2$}

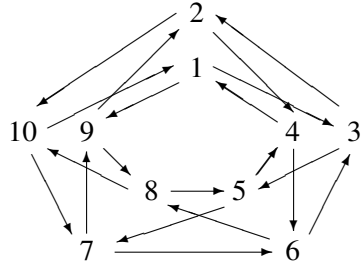
\begin{figure}[h]
	\begin{center}
	%\fbox{
	\begin{picture}(135,100)
		\put(69,69){\small$1$}\put(69,90){\small$2$}
		\put(128,45){\small$3$}\put(105,45){\small$4$}
		\put(85,22.5){\small$5$}\put(105,0){\small$6$}
		\put(28,0){\small$7$}\put(52,22.5){\small$8$}
		\put(28,45){\small$9$}\put(1,45){\small$10$}
		\put(62,26){\vector(1,0){20}}%8->5
		\put(41,3){\vector(1,0){59}}%7->6
		\put(64,93){\vector(-3,-2){52}}%2->10
		\put(35,57){\vector(1,1){30}}%9->2
		\put(65,67){\vector(-2,-1){27}}%9->1
		\put(16,50){\vector(2,1){47}}%10->1
		\put(102,52){\vector(-3,2){23}}%1->4
		\put(78,86){\vector(1,-1){30}}%4->2
		\put(125,58){\vector(-4,3){46}}%3->2
		\put(78,73){\vector(2,-1){45}}%1->3
		\put(46,27){\vector(-2,1){30}}%8-10
		\put(10,40){\vector(1,-2){15}}%7-10
		\put(30.5,10){\vector(0,1){31}}%7-9
		\put(37,42){\vector(1,-1){10}}%8-9
		\put(125,42){\vector(-2,-1){30}}%5-3
		\put(94,32){\vector(3,4){8}}%4-5
		\put(108,42){\vector(0,-1){31}}%4-6
		\put(114,10){\vector(1,2){15}}%3-6
		\put(82,20){\vector(-3,-1){40}}%7-5
		\put(99,8){\vector(-3,1){38}}%8-6
	\end{picture}
	%}
	\caption{Quiver $Q_{104}$}\label{Fig:deg_q-Garnier_104}
	\end{center}
\end{figure}

In a similar manner to Section \ref{Sec:Affine_Weyl}, we obtain simple reflections
\begin{align*}
	&r_i = \mu_{2i+1}(2i+1,2i+2)\mu_{2i+1}\quad (i=0,\ldots,4), \\
	&s_0 = \mu_1\mu_4\mu_5\mu_8(8,9)\mu_8\mu_5\mu_4\mu_1,\quad
	s_1 = \mu_2\mu_3\mu_7\mu_{10}(10,6)\mu_{10}\mu_7\mu_3\mu_2,
\end{align*}
simple roots
\[
	\alpha_i = y_{2i+1}y_{2i+2}\quad (i=0,\ldots,4),\quad
	\beta_0 = y_1y_4y_5y_8y_9,\quad
	\beta_1 = y_2y_3y_6y_7y_{10}
\]
and a null root
\[
	q = \prod_{i=0}^{4}\alpha_i = \beta_0\beta_1 = \prod_{i=1}^{10}y_i
\]
from the quiver $Q_{104}$.
In addition, we introduce Dynkin diagram automorphisms
\begin{align*}
	&\pi_1 = (1,4,5,8,9)(2,3,6,7,10), \\
	&\pi_2 = (1,2)(3,4)(5,6)(7,8)(9,10), \\
	&\pi_3 = \iota(3,10)(4,9)(5,8)(6,7).
\end{align*}
Then their actions on the coefficients, the simple roots and the null root have the same formulas as those given by Kajiwara, Noumi and Yamada in \cite{KNY}.
We omit them here.

\begin{fact}[\cite{KNY}]
The groups $\langle r_0,\ldots,r_4\rangle$ and $\langle s_0,s_1\rangle$ are isomorphic to the affine Weyl groups of type $A^{(1)}_4$ and $A^{(1)}_1$ respectively.
Moreover, those two groups are mutually commutative.
\end{fact}

\begin{thm}
Through the confluence $11\to2$, the simple reflections and the Dynkin diagram automorphisms are reduced as follows.
\[\begin{array}{|c||c|c|c|c|c|c|c|c|c|c|}\hline
	Q_{11} & r_0 & r_1 & r_2 & r_3 & r_4 & s_0 & s_1 & \pi_1^6 & \pi_1^5 & \pi_2 \\\hline
	Q_{104} & r_0 & r_1 & r_2 & r_3 & r_4 & s_0 & s_1 & \pi_1 & \pi_2 & \pi_3 \\\hline
\end{array}\]
The simple roots are also reduced as follows.
\[\begin{array}{|c||c|c|c|c|c|c|c|}\hline
	Q_{11} & \alpha_0 & \alpha_1 & \alpha_2 & \alpha_3 & \alpha_4 & \beta_0 & \beta_1 \\\hline
	Q_{104} & \alpha_0 & \alpha_1 & \alpha_2 & \alpha_3 & \alpha_4 & \beta_0 & \beta_1 \\\hline
\end{array}\]
\end{thm}

We can prove this theorem in a similar manner to Section \ref{Sec:Degeneration_12_11}, whose detail we omit here.

The translation subgroup of $\langle r_0,\ldots,r_4,s_0,s_1,\pi_1,\pi_2,\pi_3\rangle$ is generated by
\begin{align*}
	&T_i = r_{i+1}r_{i+2}r_{i+3}r_{i+4}r_{i+3}r_{i+2}r_{i+1}r_i\quad (i=0,\ldots,4), \\
	&U_0 = s_0s_1,\quad
	U_1 = s_1s_0,\quad
	V = \pi_1r_4\ldots r_1,\quad
	V' = \pi_2s_1,
\end{align*}
where the indices of $r_i$ are congruent modulo $5$.
They act on the simple roots as
\begin{align*}
	&T_i(\alpha_j) = q^{-\delta_{j,i-1}+2\delta_{j,i}-\delta_{j,i+1}}\alpha_j,\quad
	T_i(\beta_l) = \beta_l, \\
	&U_k(\alpha_j) = \alpha_j,\quad
	U_k(\beta_l) = q^{4\delta_{l,k}-2}\beta_l, \\
	&V(\alpha_j) = q^{\delta_{j,0}-\delta_{j,1}}\alpha_j,\quad
	V(\beta_l) = \beta_l, \\
	&V'(\alpha_j) = \alpha_j,\quad
	V'(\beta_l) = q^{\delta_{l,0}-\delta_{l,1}}\beta_l
\end{align*}
for $i,j=0,\ldots,4$ and $k,l=0,1$, where $\delta_{i,j}$ is the Kronecker delta whose indices are congruent modulo $5$.

\begin{cor}
Through the confluence $11\to2$, the translations are reduced as follows.
\[\begin{array}{|c||c|c|c|c|c|c|c|c|c|}\hline
	Q_{11} & T_0 & T_1 & T_2 & T_3 & T_4 & U_0 & U_1 & V'VU_1 & V' \\\hline
	Q_{104} & T_0 & T_1 & T_2 & T_3 & T_4 & U_0 & U_1 & V & V' \\\hline
\end{array}\]
\end{cor}

Note that the translation $\tau_c$ in $Q_{11}$ is reduced to that $\tau_c=V'U_1=\pi_2s_0$ in $Q_{104}$ through the confluence $11\to2$.

%%%%%%%%%%%%%%%%%%%%%%%%%%%%%%%%%%%%%%%%%%%%%%%%%%%%%%%%%%%%%%%%
\subsection{Confluence $1\to11$}

\begin{figure}[h]
	\begin{center}
	%\fbox{
	\begin{picture}(135,100)
	\put(90,80){\small$1$}\put(116,80){\small$2$}
	\put(116,40){\small$3$}\put(90,40){\small$4$}
	\put(65,18){\small$5$}\put(65,0){\small$6$}	
	\put(13,40){\small$7$}\put(40,40){\small$8$}
	\put(40,80){\small$9$}\put(10,80){\small$10$}
	\put(75,28){\vector(1,1){11}}%3->5
	\put(50,38){\vector(1,-1){11}}%5->8
	\put(92,37){\vector(-2,-3){18}}%4-6
	\put(60,22){\vector(-2,1){38}}%7-5
	\put(112,40){\vector(-2,-1){37}}%3-5
	\put(62,10){\vector(-2,3){18}}%6-8
	\put(75,5){\vector(4,3){41}}%6-3
	\put(19,37){\vector(4,-3){42}}%7-6
	\put(15.5,76){\vector(0,-1){25}}%7-10
	\put(42.5,75){\vector(0,-1){25}}%8-9
	\put(92.5,52){\vector(0,1){25}}%1-4
	\put(118.5,52){\vector(0,1){25}}%2-3
	\put(36,51){\vector(-2,3){16}}%8-10
	\put(20,51){\vector(2,3){16}}%7-9
	\put(97,76){\vector(2,-3){16}}%1-3
	\put(113,76){\vector(-2,-3){16}}%2-3
	\end{picture}
	%}
	\caption{Quiver $Q_{105}$}\label{Fig:deg_q-Garnier_105}
	\end{center}
\end{figure}
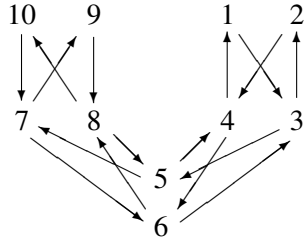

We change the vertex $11$ to $1$ after the confluence.
In a similar manner to Section \ref{Sec:Affine_Weyl}, we obtain simple reflections
\[
	r_i = \mu_{2i-1}(2i-1,2i)\mu_{2i-1}\quad (i=1,\ldots,5)
\]
and simple roots
\[
	\alpha_i = y_{2i-1}y_{2i}\quad (i=1,\ldots,5)
\]
from the quiver $Q_{105}$.
Then their actions on the coefficients and the simple roots have the same formula as those given by Inoue, Lam and Pylyavskyy in \cite{ILP}.
We omit them here.

\begin{fact}[\cite{ILP}]
The group $\langle r_1,\ldots,r_5\rangle$ is isomorphic to the finite Weyl group of type $A_5$.
\end{fact}

In addition, we can introduce Dynkin diagram automorphisms
\[
	\pi_1 = (1,2)(3,4)(5,6)(7,8)(9,10),\quad
	\pi_2 = \iota(3,4)(7,8),\quad
	\pi_3 = \iota(1,10)(2,9)(3,8)(4,7)(5,6)
\]
and a variable
\[
	\gamma = \frac{y_1y_5y_9}{y_2y_6y_{10}}.
\]
Similarly to Section \ref{Sec:Degeneration_12_11}, we call $\gamma$ a simple root for the sake of convenience.
Then $\gamma$ is invariant under the action of the simple reflections and the Dynkin diagram automorphisms act on the simple roots as
\begin{align*}
	&\pi_1(\alpha_i) = \alpha_i\quad (i=1,\ldots,5),\quad
	\pi_1(\gamma) = \frac{1}{\gamma}, \\
	&\pi_2(\alpha_i) = \frac{1}{\alpha_i}\quad (i=1,\ldots,5),\quad
	\pi_2(\gamma) = \frac{1}{\gamma}, \\
	&\pi_3(\alpha_i) = \frac{1}{\alpha_{6-i}}\quad (i=1,\ldots,5),\quad
	\pi_3(\gamma) = \gamma.
\end{align*}
However, we can not give any translation from the group $\langle r_1,\ldots,r_5,\pi_1,\pi_2\rangle$.

\begin{rem}
Through the confluence $1\to11$, the simple reflections are reduced as follows.
\[\begin{array}{|c||c|c|c|}\hline
	Q_{11}  & r_1 & r_2 & r_3 \\\hline
	Q_{105} & r_2 & r_3 & r_4 \\\hline
\end{array}\]
We predict that the simple reflections $r_1,r_5$ and the Dynkin diagram automorphisms $\pi_1,\pi_2,\pi_3$ can not be obtained through the confluence $1\to11$ due to the form of $Q_{105}$.
\end{rem}

%%%%%%%%%%%%%%%%%%%%%%%%%%%%%%%%%%%%%%%%%%%%%%%%%%%%%%%%%%%%%%%%
% Section 7
%%%%%%%%%%%%%%%%%%%%%%%%%%%%%%%%%%%%%%%%%%%%%%%%%%%%%%%%%%%%%%%%
\section{Degeneration of the particular solution}\label{Sec:HG_Solution}

In this section, we focus on the quivers $Q_{11}$, $Q_{101}$ and $Q_{102}$ and the translation $\tau_c$.
Then the corresponding evolution directions of the degenerate $q$-Garnier systems admit particular solutions in terms of the basic hypergeometric series ${}_2\phi_2$ and ${}_1\phi_2$.
Recall that the action of $\tau_c$ is divergent through the confluence $Q_{11}\to Q_{103}$.
Also note that the confluences $Q_{11}\to Q_{104}$ and $Q_{11}\to Q_{105}$ are not consistent with the condition for the particular solution $y_3=y_7=y_{11}=-1$.

%%%%%%%%%%%%%%%%%%%%%%%%%%%%%%%%%%%%%%%%%%%%%%%%%%%%%%%%%%%%%%%%
\subsection{Quiver $Q_{11}$}

Consider the translation $\tau_c=V'U_1=\pi_1^5s_0$ in $Q_{11}$.
It acts on the simple roots as
\[
	\tau_c(\alpha_i) = \alpha_i\quad (i=0,\ldots,4),\quad
	\tau_c(\beta_0) = q^{-1}\beta_0,\quad
	\tau_c(\beta_1) = q\beta_1,\quad
	\tau_c(\gamma) = q^5\gamma.
\]
Assume that
\[
	y_3 = -1,\quad
	y_7 = -1,\quad
	y_{11} = -1.
\]
Then the action of $\tau_c$ on $y_1,y_5,y_9$ provides a $q$-Riccati system, whose detail we omit here.
On the other hand, we consider a system of linear $q$-difference equations
\begin{equation}\label{eq:HGE_11}
	x(q^{-1}t) = \left(A_0+tA_1\right)x(t)
\end{equation}
with $3\times3$ matrices
\begin{align*}
	&A_0 = \begin{bmatrix}1&0&0\\(1-q)\alpha_0&\alpha_0&0\\(1-q)\alpha_0&\alpha_0(1-\alpha_1)&\alpha_0\alpha_1\alpha_2\end{bmatrix}, \\
	&A_1 = \begin{bmatrix}-\alpha_0&-\frac{\alpha_0(1-\alpha_1)}{1-q}&-\frac{\alpha_0\alpha_1\alpha_2(1-\alpha_3)}{1-q}\\0&0&0\\0&0&0\end{bmatrix}.
\end{align*}
A solution to system \eqref{eq:HGE_11} gives that to the $q$-Riccati system.

\begin{prop}
The evolution direction of the degenerate $q$-Garnier system provided by $\tau_c$ in $Q_{11}$ admits a particular solution
\begin{align*}
	&y_1 = -\frac{x_0}{x_1},\quad
	y_2 = \frac{\alpha_0x_1}{x_0},\quad
	y_3 = -1,\quad
	y_4 = -\alpha_1,\quad
	y_5 = -\frac{x_1}{x_2},\quad
	y_6 = -\frac{\alpha_2x_2}{x_1}, \\
	&y_7 = -1,\quad
	y_8 = -\alpha_3,\quad
	y_9 = \frac{qtx_2}{x_0},\quad
	y_{10} = \frac{\alpha_4x_0}{qtx_2},\quad
	y_{11} = -1
\end{align*}
with
\begin{align*}
	&x_0 = {}_2\phi_2\left[\begin{array}{c}\alpha_0\alpha_1,\alpha_0\alpha_1\alpha_2\alpha_3\\\alpha_0,\alpha_0\alpha_1\alpha_2\end{array};q,q\alpha_0t\right], \\
	&x_1 = \frac{(1-q)\alpha_0}{1-\alpha_0}{}_2\phi_2\left[\begin{array}{c}\alpha_0\alpha_1,\alpha_0\alpha_1\alpha_2\alpha_3\\q\alpha_0,\alpha_0\alpha_1\alpha_2\end{array};q,q^2\alpha_0t\right], \\
	&x_2 = \frac{(1-q)\alpha_0(1-\alpha_0\alpha_1)}{(1-\alpha_0)(1-\alpha_0\alpha_1\alpha_2)}{}_2\phi_2\left[\begin{array}{c}q\alpha_0\alpha_1,\alpha_0\alpha_1\alpha_2\alpha_3\\q\alpha_0,q\alpha_0\alpha_1\alpha_2\end{array};q,q^2\alpha_0t\right],
\end{align*}
where
\[
	{}_2\phi_2\left[\begin{array}{c}a_1,a_2\\b_1,b_2\end{array};q,t\right] = \sum_{n=0}^{\infty}\frac{(a_1;q)_n(a_2;q)_n}{(b_1;q)_n(b_2;q)_n(q;q)_n}(-1)^nq^{\frac{n(n-1)}{2}}t^n.
\]
\end{prop}

System \eqref{eq:HGE_11} is derived from system \eqref{eq:HGE_12} by a replacement
\[
	\alpha_0 \to \varepsilon,\quad
	\alpha_5 \to \frac{\alpha_0}{\varepsilon},\quad
	t \to \varepsilon t,\quad
	x_1 \to -\frac{x_1}{\varepsilon(1-q)},\quad
	x_2 \to -\frac{x_2}{\varepsilon(1-q)}
\]
and a limit $\varepsilon\to0$.
This confluence procedure is naturally derived from that given in Section \ref{Sec:Degeneration_12_11}.

%%%%%%%%%%%%%%%%%%%%%%%%%%%%%%%%%%%%%%%%%%%%%%%%%%%%%%%%%%%%%%%%
\subsection{Quiver $Q_{101}$}

Consider the translation $\tau_c=V'U_1=\pi_2\pi_1^2s_0$ in $Q_{101}$.
It acts on the simple roots as
\[
	\tau_c(\alpha_i) = \alpha_i\quad (i=0,\ldots,3),\quad
	\tau_c(\beta_0) = q^{-1}\beta_0,\quad
	\tau_c(\beta_1) = q\beta_1,\quad
	\tau_c(\gamma) = q^2\gamma.
\]
Assume that
\[
	y_3 = -1,\quad
	y_4 = -1,\quad
	y_7 = -1.
\]
Then the action of $\tau_c$ on $y_1,y_5,y_9$ provides a $q$-Riccati system, whose detail we omit here.
On the other hand, we consider a system of linear $q$-difference equations
\begin{equation}\label{eq:HGE_101}
	x(q^{-1}t) = \left(A_0+tA_1\right)x(t)
\end{equation}
with $3\times3$ matrices
\begin{align*}
	&A_0 = \begin{bmatrix}1&0&0\\(1-q)\alpha_0&\alpha_0&0\\0&(1-q)\alpha_0\alpha_1&\alpha_0\alpha_1\end{bmatrix}, \\
	&A_1 = \begin{bmatrix}0&\frac{\alpha_0\alpha_1}{1-q}&\frac{\alpha_0\alpha_1(1-\alpha_2)}{(1-q)^2}\\0&0&0\\0&0&0\end{bmatrix}.
\end{align*}
A solution to system \eqref{eq:HGE_101} gives that to the $q$-Riccati system.

\begin{prop}
The evolution direction of the degenerate $q$-Garnier system provided by $\tau_c$ in $Q_{101}$ admits a particular solution
\begin{align*}
	&y_1 = -\frac{x_0}{x_1},\quad
	y_2 = \frac{\alpha_0x_1}{x_0},\quad
	y_3 = -1,\quad
	y_4 = -1,\quad
	y_5 = -\frac{x_1}{x_2},\quad
	y_6 = \frac{\alpha_1x_2}{x_1}, \\
	&y_7 = -1,\quad
	y_8 = -\alpha_2,\quad
	y_9 = -\frac{qtx_2}{x_0},\quad
	y_{10} = -\frac{\alpha_3x_0}{qtx_2}
\end{align*}
with
\begin{align*}
	&x_0 = {}_1\phi_2\left[\begin{array}{c}\alpha_0\alpha_1\alpha_2\\\alpha_0,\alpha_0\alpha_1\end{array};q,q\alpha_0^2\alpha_1t\right], \\
	&x_1 = \frac{(1-q)\alpha_0}{1-\alpha_0}{}_1\phi_2\left[\begin{array}{c}\alpha_0\alpha_1\alpha_2\\\alpha_0,\alpha_0\alpha_1\end{array};q,q^2\alpha_0^2\alpha_1t\right], \\
	&x_2 = \frac{(1-q)^2\alpha_0^2\alpha_1}{(1-\alpha_0)(1-\alpha_0\alpha_1)}{}_1\phi_2\left[\begin{array}{c}\alpha_0\alpha_1\alpha_2\\q\alpha_0,q\alpha_0\alpha_1\end{array};q,q^3\alpha_0^2\alpha_1t\right],
\end{align*}
where
\[
	{}_1\phi_2\left[\begin{array}{c}a_1\\b_1,b_2\end{array};q,t\right] = \sum_{n=0}^{\infty}\frac{(a_1;q)_n}{(b_1;q)_n(b_2;q)_n(q;q)_n}q^{n(n-1)}t^n.
\]
\end{prop}

System \eqref{eq:HGE_101} is derived from system \eqref{eq:HGE_11} by a replacement
\[
	\alpha_1 \to \frac{\alpha_1}{\varepsilon},\quad
	\alpha_2 \to \varepsilon,\quad
	\alpha_3 \to \alpha_2,\quad
	\alpha_4 \to \alpha_3,\quad
	t \to \varepsilon t,\quad
	x_2 \to -\frac{x_2}{\varepsilon(1-q)}
\]
and a limit $\varepsilon\to0$.
This confluence procedure is naturally derived from that given in Section \ref{Sec:Degeneration_11_101}.

%%%%%%%%%%%%%%%%%%%%%%%%%%%%%%%%%%%%%%%%%%%%%%%%%%%%%%%%%%%%%%%%
\subsection{Quiver $Q_{102}$}

Consider the translation $\tau_c=V=\pi_3\pi_2\pi_1\pi_3\pi_2\pi_1$ in $Q_{102}$.
It acts on the simple roots as
\[
	\tau_c(\alpha_i) = \alpha_i\quad (i=0,\ldots,3),\quad
	\tau_c(\gamma_1) = q^2\gamma_1,\quad
	\tau_c(\gamma_2) = \gamma_2.
\]
Assume that
\[
	y_3 = -1,\quad
	y_6 = -1,\quad
	y_7 = -1.
\]
Then the action of $\tau_c$ on $y_1,y_5,y_9$ provides a $q$-Riccati system, whose detail we omit here.
On the other hand, we consider a system of linear $q$-difference equations
\begin{equation}\label{eq:HGE_102}
	x(q^{-1}t) = \left(A_0+tA_1\right)x(t)
\end{equation}
with $3\times3$ matrices
\begin{align*}
	&A_0 = \begin{bmatrix}1&0&0\\(1-q)\alpha_0&\alpha_0&0\\(1-q)^2\alpha_0&(1-q)\alpha_0&\alpha_0\alpha_1\end{bmatrix}, \\
	&A_1 = \begin{bmatrix}-\alpha_0&-\frac{\alpha_0}{1-q}&-\frac{\alpha_0\alpha_1(1-\alpha_2)}{(1-q)^2}\\0&0&0\\0&0&0\end{bmatrix}.
\end{align*}
A solution to system \eqref{eq:HGE_102} gives that to the $q$-Riccati system.

\begin{prop}
The evolution direction of the degenerate $q$-Garnier system provided by $\tau_c$ in $Q_{102}$ admits a particular solution
\begin{align*}
	&y_1 = -\frac{x_0}{x_1},\quad
	y_2 = \frac{\alpha_0x_1}{x_0},\quad
	y_3 = -1,\quad
	y_4 = \frac{\alpha_1x_2}{x_1},\quad
	y_5 = -\frac{x_1}{x_2},\quad
	y_6 = -1, \\
	&y_7 = -1,\quad
	y_8 = -\alpha_2,\quad
	y_9 = -\frac{qtx_2}{x_0},\quad
	y_{10} = -\frac{\alpha_3x_0}{qtx_2}
\end{align*}
with
\begin{align*}
	&x_0 = {}_2\phi_2\left[\begin{array}{c}0,\alpha_0\alpha_1\alpha_2\\\alpha_0,\alpha_0\alpha_1\end{array};q\alpha_0t\right], \\
	&x_1 = \frac{(1-q)\alpha_0}{1-\alpha_0}{}_2\phi_2\left[\begin{array}{c}0,\alpha_0\alpha_1\alpha_2\\q\alpha_0,\alpha_0\alpha_1\end{array};q^2\alpha_0t\right], \\
	&x_2 = \frac{(1-q)^2\alpha_0}{(1-\alpha_0)(1-\alpha_0\alpha_1)}{}_2\phi_2\left[\begin{array}{c}0,\alpha_0\alpha_1\alpha_2\\q\alpha_0,q\alpha_0\alpha_1\end{array};q^2\alpha_0t\right].
\end{align*}
\end{prop}

System \eqref{eq:HGE_102} is derived from system \eqref{eq:HGE_11} by a replacement
\[
	\alpha_1 \to \varepsilon,\quad
	\alpha_2 \to \frac{\alpha_1}{\varepsilon},\quad
	\alpha_3 \to \alpha_2,\quad
	\alpha_4 \to \alpha_3,\quad
	x_2 \to \frac{x_2}{1-q}
\]
and a limit $\varepsilon\to0$.
This confluence procedure is naturally derived from that given in Section \ref{Sec:Degeneration_11_102}.

%%%%%%%%%%%%%%%%%%%%%%%%%%%%%%%%%%%%%%%%%%%%%%%%%%%%%%%%%%%%%%%%
\section*{Acknowledgement}

The authors would like to express their gratitude to Professors Tetsu Masuda, Naoto Okubo and Teruhisa Tsuda for helpful comments and advice.
This work is supported by JSPS KAKENHI Grant Number 25KJ0371.
The authors declare no conflicts of interest associated with this manuscript.

\section*{Data Availability}

Data sharing not applicable to this article as no datasets were generated or analyzed during the current study.

%%%%%%%%%%%%%%%%%%%%%%%%%%%%%%%%%%%%%%%%%%%%%%%%%%%%%%%%%%%%%%%%

\end{document}